\documentclass[final]{siamart250211}

\usepackage{amsmath,amsfonts,amssymb}
\usepackage{graphicx}
\usepackage{nicefrac}
\usepackage{MnSymbol}
\usepackage{comment}
\usepackage{pgfplots}
\usepackage{xfrac}
\usepackage{xspace}
\usepackage[normalem]{ulem}
\usepackage{pifont}

\usepackage{tikz}
\usetikzlibrary{external}
\tikzsetexternalprefix{tikz_figures/}
\usepackage{subcaption}
\usepackage{pgfplots}

\usepackage{accents}

\usepackage{multirow}

\newcommand{\ie}{\textit{i.\,e\mbox{.}}\xspace}
\numberwithin{equation}{section}

\newcommand{\IR}{\mathbb{R}}

\newtheorem{remark}{Remark}

\makeatletter
\@ifpackageloaded{stix}{%
}{%
	\DeclareFontEncoding{LS2}{}{\noaccents@}
	\DeclareFontSubstitution{LS2}{stix}{m}{n}
	\DeclareSymbolFont{stix@largesymbols}{LS2}{stixex}{m}{n}
	\SetSymbolFont{stix@largesymbols}{bold}{LS2}{stixex}{b}{n}
	\DeclareMathDelimiter{\lBrace}{\mathopen} {stix@largesymbols}{"E8}%
	{stix@largesymbols}{"0E}
	\DeclareMathDelimiter{\rBrace}{\mathclose}{stix@largesymbols}{"E9}%
	{stix@largesymbols}{"0F}
} %
\makeatother
\usepackage{stmaryrd}
\usepackage{mathtools}
\newcommand{\bvec}{\boldsymbol}

\DeclareMathOperator{\Grad}{\bvec{\nabla}}

\DeclareMathOperator{\Div}{div}

\newcommand{\OmG}{{\Omega\setminus\Gamma}}
\newcommand{\mean}[1]{\left\lBrace #1 \right\rBrace}
\newcommand{\jump}[1]{\big\llbracket #1 \big\rrbracket}
\newcommand{\jumpt}[1]{\llbracket #1 \rrbracket}
\newcommand{\ess}{\operatorname{ess}}
\newcommand{\tr}{\operatorname{tr}}

\newcommand{\vd}{\bvec{d}}
\newcommand{\vf}{\bvec{f}}
\newcommand{\vn}{\bvec{n}}
\newcommand{\vq}{{\bvec{q}}}
\newcommand{\vw}{{\bvec{w}}}
\newcommand{\vx}{\bvec{x}}

\newcommand{\zerobf}{\boldsymbol{0}}
\newcommand{\blf}[1]{\mathsf{#1}}
\newcommand{\cF}{\mathcal{F}}
\newcommand{\cM}{\mathcal{M}}
\newcommand{\cQ}{\mathcal{Q}}
\newcommand{\cT}{\mathcal{T}}
\newcommand{\cW}{\mathcal{W}}
\newcommand{\pitr}{\operatorname{\Pi tr}}

\usepackage{mathtools}

\makeatletter
\def\@tvsp{\mathchoice{{}\mkern-4.5mu}{{}\mkern-4.5mu}{{}\mkern-2.5mu}{}}
\def\ltrivert{\left|\@tvsp\left|\@tvsp\left|}
\def\rtrivert{\right|\@tvsp\right|\@tvsp\right|}
\makeatother
\newcommand\tnorm[1]{\ltrivert#1\rtrivert}

\newcommand\pnorm[1]{\tnorm{#1}_\tau}

\usepackage{placeins}

\usepackage{hyperref}
\hypersetup{
	pdftitle={On a Hybrid Mixed Domain Decomposition Method},
	pdfauthor={Kersten Schmidt, Timon Seibel, Sebastian Schöps},
	pdfsubject={Finite element method, domain decomposition, numerical analysis, mixed formulation, hybrid method},
	pdfkeywords={HMDD, domain decomposition, finite element method, numerical analysis, mixed formulation, hybrid method},
	pdflang={en-US},
}

\begin{document}
\title{On a hybrid mixed domain decomposition method}

\author{
Kersten Schmidt\thanks{Department of Mathematics, Technical University of Darmstadt, Dolivostr. 15, 64293 Darmstadt, Germany} \and
Timon Seibel\thanks{Computational Electromagnetics Group, Technical University of Darmstadt, Schloßgartenstr. 8, 64289 Darmstadt, Germany} \and
Sebastian Schöps$^\dagger$
}

\maketitle

\begin{abstract}
   We present a domain decomposition formulation based on hybridization which is inspired by hybridized discontinuous Galerkin (HDG) methods, that enhance mixed domain decomposition methods by incorporating stabilization terms.
   Unlike discontinuous Galerkin methods, our analysis of the proposed finite element method is based on a corresponding consistent variational formulation and a perturbed Galerkin method. In the variational formulation the divergence appears not only within subdomains, but also as an $L^2$-surface quantity on the interfaces. Furthermore, the traces of the finite element functions on the interfaces are replaced by $L^2$-distributions. The well-posedness of the perturbed Galerkin method is shown for an appropriate choice of subspaces, in a manner similar to that of the variational formulation. %
   For the finite element method we use Raviart-Thomas elements for the dual variable and piecewise polynomials for the primal and hybrid variables, respectively. %
   We perform an analysis of the discretization error which is explicit in the stabilization parameter $\tau$.
   Numerical experiments for piecewise smooth solutions using finite element spaces of order~$q$ on curved quadrilateral meshes confirm the predicted convergence rate of $q+1$ for small values of $\tau$.
   In the error analysis we observe the discretization error to be uniformly bounded in $\tau$. %
   Even for large $\tau$ values the observed convergence rates for the primal and for the hybrid variables are $q+1$. For the dual variable the convergence rate depends on the stabilization parameter and the mesh-width, with an asymptotic rate of $q+\nicefrac12$.
\end{abstract}



\pagestyle{myheadings}
\thispagestyle{plain}
\markboth{}{On a hybrid mixed domain decomposition method}

\section{Introduction}
The aim of this paper is to introduce and analyze a finite element (FE) formulation on a domain partitioned into subdomains, where the coupling of the subdomains is inspired by \emph{Hybridized Discontinuous Galerkin} (HDG) methods. Throughout this paper, we will consider as toy model the Poisson equation with homogeneous Dirichlet boundary conditions:
Let $f\in L^2(\Omega)$. Find $u$ such that
\begin{subequations}
	\begin{align}
		-\Div \kappa \Grad u &= f, \text{ in } \Omega,\label{Int:eq:PoisProbA}\\
		u &= 0, \text{ on } \partial\Omega,\label{Int:eq:PoisProbB}
	\end{align}\label{Int:eq:PoisProb}
\end{subequations}
\noindent%
where $\kappa$ is a positive function on $\Omega\subset\mathbb{R}^d,\, d=2,\,3$. %
In particular, we will focus on the mixed formulation of the Poisson problem \eqref{Int:eq:PoisProb},
where a flux variable $\bvec{q} = \kappa\Grad u$ is introduced as an additional unknown. The classical finite element methods for the mixed Poisson problem based on Raviart-Thomas or Brezzi-Douglas-Marini spaces are looking for an approximation $\bvec{q}_h$ of the flux variable that is $H(\Div,\Omega)$-conforming, \ie, the normal component of $\bvec{q}_h$ is continuous across element boundaries. As pointed out by Cockburn and Gopalakrishnan \cite{Cockburn.Gopalakrishnan:2004}, a disadvantage of these methods is that their system matrix is not positive definite. Further, trying to eliminate the degrees of freedom (dofs) of $\bvec{q}_h$ in order to obtain a Schur-complement w.r.t. $u_h$ requires the calculation of inverses of non-sparse matrices, which is numerically challenging. To overcome these difficulties, Cockburn and Gopalakrishnan consider an additional, so-called \emph{hybrid}, variable $\lambda_h$ that approximates the numerical trace of $u_h$ on the edges between elements \cite{Cockburn.Gopalakrishnan:2004}. This approach was inspired by Fraejis de Veubeke who first introduced such a $\lambda_h$ as a Lagrange multiplier in 1965 \cite{FraejisDeVeubeke:1965}. It is then possible to rearrange the resulting matrix equation for $\lambda_h$ in order to obtain a Schur-complement system whose system matrix is symmetric and positive definite. One can show that the Schur-complement system has significantly less d.o.f.s than the original system and, further, it can be concluded that the d.o.f.s related to $\lambda_h$ are the only globally coupled ones \cite{Cockburn.Gopalakrishnan:2004, Cockburn.Gopalakrishnan.Lazarov:2009}. %
In \cite{Cockburn.Gopalakrishnan.Lazarov:2009}, Cockburn, Gopalakrishnan and Lazarov present a \emph{unified framework} for HDG methods, in which an approximation to $\bvec{q}_h\cdot\bvec{n}$ is introduced by \[\hat{\bvec{q}}_h^\pm\cdot\bvec{n}^\pm \coloneq \bvec{q}^\pm\cdot\bvec{n}^\pm + \tau^\pm(u_h-\hat{u}_h^\pm),\]
where $\hat{u}_h^\pm$ is the numerical trace and $\tau^\pm$ is a non-negative function on the edges, which is referred to as \emph{stabilization function}. Note, that the superscript ${}^\pm$ added to a function shall indicate that it is a double-valued function on the element edges. A common choice for $\tau^\pm$, which is the one we use in this paper, is the \emph{Lehrenfeld-Schöberl} stabilization function \cite{Cockburn:2016} \[\tau^\pm(u_h-\hat{u}_h^\pm) \coloneq \pm\tau (u_h - \hat{u}_h^\pm),\] where a multiplicative constant $\tau\geq0$ is used. It was shown \cite{Cockburn:2016, Cockburn.Gopalakrishnan.Lazarov:2009} that for $0<\tau<\infty$ non-conforming meshes are an admissible choice, which do not require any further effort. On conforming meshes $\tau$ can be chosen as zero, leading to the \emph{hybrid mixed methods} presented in \cite{Cockburn.Gopalakrishnan:2004}, or formally as $\tau=\infty$ resulting in hybridizable continuous Galerkin methods. %
Mixed methods have been used for discretization on subdomains with Lagrange multipliers on their interfaces, see e.g.~\cite{Arnold.Brezzi:1985, Glowinski.Wheeler:1988, Brezzi.Fortin:1991} for matching grids, which are called \emph{mortar mixed method} when used as a mortaring technique~\cite{Arbogast.Cowsar.Wheeler.Yotov:2000}.
In~\cite{Arbogast.Pencheva.Wheeler.Yotov:2007} a hybrid variable~$\lambda$, that is in $H^{\nicefrac12}(\Gamma_{i,j})$ on each subdomain interface $\Gamma_{i,j}$, was introduced already at the continuous level. Further, the authors considered the mesh widths in the patches and on the interface separately in order to obtain improved error estimates. %
The present work studies a \emph{hybrid mixed domain decomposition} (HMDD) approach that extends the mixed discretization on subdomains by including Lehrenfeld–Schöberl stabilization terms, as known from HDG methods, with a stabilization parameter $\tau\geq0$.
We introduce and analyze a high-order finite element discretization of HMDD on curved cells for the generalized Poisson problem~\eqref{Int:eq:PoisProb} as an example of the application to elliptic PDEs. In difference to DG methods our analysis of the proposed finite element method is based on a corresponding variational formulation, where the divergence has a distributional contribution on the interface $\Gamma$ and the traces of the FE functions on $\Gamma$ are replaced by $L^2(\Gamma)$ distributions. We show inf-sup stability of the variational formulation, where the proof transfers similarly to a class of perturbed Galerkin methods and so to high order FE discretization using Raviart-Thomas elements and piecewise polynomials for the case of matching grids, which allows for a rather classical FE error analysis with a consistency error contribution. Our new analysis of the FE discretization of the HMDD method with stabilization terms is applied for conforming meshes where well-posedness is shown even if the stabilization parameter $\tau = 0$. We expect that our proof technique can serve as a basis for HMDD methods on non-matching grids similar to HDG methods.
The rest of the paper is organized as follows: the HMDD method, the main theorems of well-posedness and the estimate on the discretization error are stated in Sec.~\ref{sec:main}. The corresponding variational formulation and function spaces are introduced in Sec.~\ref{sec:HDD}, where the inf-sup conditions and so the well-posedness are proven. In Sec.~\ref{sec:Galerkin}, a class of perturbed Galerkin methods is introduced, where their inf-sup stability is shown and their discretization error is estimated by the sum of the best-approximation error and a consistency error. %
The high-order finite element method with projected traces of piecewise polynomials on the interfaces is discussed and analyzed in Sec.~\ref{sec:FEM}. %
Finally, in Sec.~\ref{sec:NumExs} numerical experiments illustrate the convergence properties of the HMDD method, also for different values of the parameter $\tau$.

\section{Hybrid domain decomposition formulation}
\label{sec:main}
In this section, we describe the derivation of a hybridized domain decomposition formulation and prove in the upcoming section that the formulation is well-posed. 

\subsection{Domain decomposition}
The aim of this work is to derive a domain decomposition method inspired by Hybridized Discontinuous Galerkin methods.
For this, we introduce a decomposition of the domain $\Omega$ into $n$ non-overlapping patches $\Omega_i$, $i=1,\ldots,n$, \ie, it holds that $\overline{\Omega}= \bigcup_i \overline{\Omega}_i$ and $\Omega_i \cap \Omega_j = \varnothing$ for $i\neq j$. %
We denote by $\Gamma = \bigcup_{i<j} \overline{\Omega}_i \cap \overline{\Omega}_j$ the skeleton, which is the union of all interfaces between patches. For the union of the patches we use the symbol $\Omega \setminus \Gamma$.
Almost everywhere on $\Gamma$ we define a unit normal vector field $\bvec{n}$ whose orientation is arbitrary but fixed (see Fig.~\ref{HDD:fig:DomDec} for an illustration). Further, we assume the material function $\kappa$ to be positive and bounded a.e., \ie, $0 < \ess\inf_{\bvec{x} \in \Omega} \kappa(\bvec{x}) \leq \ess\sup_{\bvec{x} \in \Omega} \kappa(\bvec{x}) < \infty$.

\begin{figure}[tb]
	\begin{center}
		\includegraphics{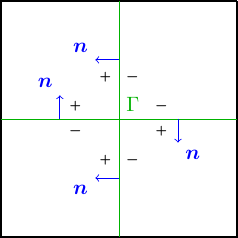}
	\end{center}
	\caption{Exemplary depiction of $\Omega$ decomposed into four disjoint subdomains. The green lines represent the interface $\Gamma$, while the blue arrows are unit normal vectors whose directions are arbitrary but fixed. Further, the $+$ and $-$ illustrate which side of each part of the interface is referred to by $\tr^+ u$ and $\tr^- u$ or $u^+$  and $u^-$, respectively.}
	\label{HDD:fig:DomDec}
\end{figure}%
If the solution $u$ of~\eqref{Int:eq:PoisProb} is smooth enough then
$(\bvec{q},u,\mu) = (\kappa \nabla u, u, \tfrac12(\tr^+u + \tr^-u))$ fulfill a strong hybrid mixed formulation for $\tau\geq0$:
\begin{subequations}
	\begin{alignat}{2}
		\tfrac{1}{\kappa}\bvec{q} - \Grad u &= 0, &&\quad\textrm{in } \OmG,\label{HDD:eq:MPHDDA}\\
		-\Div \bvec{q} &= f, &&\quad\textrm{in } \OmG,\label{HDD:eq:MPHDDB}\\
		\jump{\bvec{q}\cdot\bvec{n}} + \sum\nolimits_\pm \tau(\tr^\pm u - \mu) &= 0, &&\quad\textrm{on } \Gamma,\label{HDD:eq:MPHDDC}\\
		\tr^\pm u - \mu &= 0, &&\quad\textrm{on } \Gamma,\label{HDD:eq:MPHDDD}\\
		u &=0, &&\quad\textrm{on } \partial\Omega.\label{HDD:eq:MPHDDE}
	\end{alignat}
	\label{HDD:eq:MPHDD}
\end{subequations}
where $\tr^\pm u$ are the one-sided traces of $u$ on $\Gamma$ and $\jump{\bvec{q}\cdot\bvec{n}}$ is the jump of the normal traces of $\vq$ on $\Gamma$.
Multiplying~\eqref{HDD:eq:MPHDDA}--\eqref{HDD:eq:MPHDDC} with smooth enough functions~$\bvec{w}$ and $v$ in $\OmG$ and $\nu$ on $\Gamma$, integrating over $\OmG$ or $\Gamma$, respectively, using integration by parts in the first equations and inserting~\eqref{HDD:eq:MPHDDD}--\eqref{HDD:eq:MPHDDE} we find the weak formulation
\begin{subequations}\label{HDD:eq:HDDEqns}
\begin{align}
	\int_{\OmG} \tfrac{1}{\kappa}\bvec{q}\cdot\bvec{w} + u \ \Div\bvec{w} \ \mathrm{d}\bvec{x} + \int_{\Gamma} \mu \jump{\bvec{w}\cdot\bvec{n}} \ \mathrm{d} \sigma &= 0, \label{HDD:eq:HDD1} \\
	-\int_{\OmG} \Div\bvec{q} \, v \ \mathrm{d}\bvec{x} + \sum_\pm \int_{\Gamma} \tau (\tr^\pm u - \mu) \tr^\pm v \ \mathrm{d}\sigma &= \int_{\OmG} fv \ \mathrm{d}\bvec{x},\label{HDD:eq:HDD2} \\
	-\int_{\Gamma} \jump{\bvec{q}\cdot \bvec{n}} \nu \ \mathrm{d} \sigma - \sum_\pm \int_{\Gamma} \tau (\tr^\pm u - \mu) \nu \ \mathrm{d} \sigma &= 0.\label{HDD:eq:HDD3}
\end{align}
\end{subequations}
Here, $u$ and $\bvec{q}$ are seeked on the union of the patches $\OmG$ only and their continuity on $\Gamma$ is weakly imposed. More precisely, for well-chosen test function spaces~\eqref{HDD:eq:HDD1} leads to $\mu = \frac12 (\tr^+ u + \tr^-u)$ and $\tr^+u - \tr^-u = 0$ on $\Gamma$, and~\eqref{HDD:eq:HDD3} implies $\jump{\bvec{q}\cdot \bvec{n}} = 0$ for all $\tau \geq 0$. Now, appropriate function spaces shall be introduced such that a corresponding variational formulation is well-posed. %
\subsection{Finite element discretization on curved quadrilateral meshes}
For sake of a simple presentation we consider finite element discretizations on a conforming quadrilateral mesh $\mathcal{T}_h$ for $\Omega \subset \IR^2$. The results may be easily transferred to other conforming meshes in two or three dimensions.
First, we assume that each cell $K \in \mathcal{T}_h$ is located only in one patch. %
In this way the mesh splits naturally into meshes $\mathcal{T}_{h,i}$ of the patches and %
the edges of $\mathcal{T}_h$ on $\Gamma$ build an own mesh $\mathcal{F}_{\Gamma,h}$. The edges of $\mathcal{T}_h$
in the interior of the patches build the set $\mathcal{F}_{\Omega \setminus \Gamma,h}$.
Here, we denote by $h$ the mesh width, \ie, $h = \max_{K \in \mathcal{T}_h} \operatorname{diam}(K)$.
Each quadrilateral $K \in \mathcal{T}_h$ may be curved and is obtained as the image of a mapping $\Phi_K$ of a reference quad $\widehat{K} = [0,1]^2$ which, further, induces mappings $\Phi_F: [0,1] \to F$ for each edge $F \in \mathcal{F}_h$.
We denote by $\nabla\Phi_K$ the Jacobian matrix of $\Phi_K$.
Now, we define the finite element spaces of order~$q \geq 0$:
\begin{subequations}
\label{HDD:eq:FEspaces}
\begin{align}
   W^q(\mathcal{T}_h) &:=  \{ \bvec{w}_h : |\operatorname{det} \nabla \Phi_K| (\nabla \Phi_K)^{-1}(\bvec{w}_h \circ \Phi_K) \in \mathcal{RT}^q \,\forall K \in \mathcal{T}_h,\\
   &\hspace{4em}\nonumber
   \jump{\bvec{w}_h \cdot \bvec{n}}_F = 0 \,\forall F \in \mathcal{F}_{\Omega \setminus \Gamma,h} \}, \\
   Q^q(\mathcal{T}_h) &:=  \{        v_h : |\det \nabla \Phi_K| (v_h \circ \Phi_K) \in \mathcal{Q}^q \,\forall K \in \mathcal{T}_h \}, \\
   M^q(\mathcal{F}_{\Gamma,h}) &:=  \{  \nu_h : |\Phi_F'|(\nu_h\circ \Phi_F) \in \mathcal{P}^q \,\forall F \in \mathcal{F}_{\Gamma,h} \},
\end{align}
\end{subequations}
where $\mathcal{P}^q$ is the polynomial space of order~$q$, $\mathcal{RT}^q = (\mathcal{P}^{q} \otimes \mathcal{P}^{q+1}) \times (\mathcal{P}^{q+1} \otimes \mathcal{P}^q)$ is the Raviart-Thomas space of order~$q$, %
$\mathcal{Q}^q = \mathcal{P}^q \otimes \mathcal{P}^q$ is the tensor-product space of polynomials of order~$q$ in each local direction and $\jump{\bvec{w}_h \cdot \bvec{n}}_F$ extends the definition of the normal jump to edges of the mesh.
In this way, $\Div W^q(\mathcal{T}_h) \subseteq Q^q(\mathcal{T}_h)$ and the divergence maps the subspace of $W^q(\mathcal{T}_h)$ with vanishing normal trace on $\Gamma$ onto $Q^q(\mathcal{T}_h)$. Moreover, $\jump{W^q(\mathcal{T}_h)\cdot \bvec{n}} = M^q(\mathcal{F}_{\Gamma,h})$. See~\cite{Arnold.Boffi.Falk:2005} for a definition of $W^q(\mathcal{T}_h)$. %
Now, we can formulate the finite element formulation: \\%
Find $(\bvec{q}_h,u_h,\mu_h)\in W^q(\mathcal{T}_h) \times Q^q(\mathcal{T}_h) \times M^q(\mathcal{F}_{\Gamma,h})$ such that
\begin{subequations}
	\label{HDD:eq:dHDD:FEM:proj}
		\begin{align}
			\int_{\OmG} \tfrac{1}{\kappa}\bvec{q}_h\cdot\bvec{w}_h + u_h \ \Div\bvec{w}_h \ \mathrm{d}\bvec{x} +
			\int_{\Gamma} \mu_h \jump{\bvec{w}_h\cdot\bvec{n}} \ \mathrm{d} \sigma &=0
			\label{HDD:eq:dHDD:FEM:proj:1}\\
			-\int_{\OmG} \Div\bvec{q}_h v_h \ \mathrm{d}\bvec{x} + \sum\nolimits_\pm \int_{\Gamma} \tau (\pitr^\pm u_h - \mu_h) \pitr^\pm v_h \ \mathrm{d}\sigma &= \int_{\OmG} fv_h \ \mathrm{d}\bvec{x},
			\label{HDD:eq:dHDD:FEM:proj:2}\\
			-\int_{\Gamma} \jump{\bvec{q}_h\cdot \bvec{n}} \nu_h \ \mathrm{d} \sigma - \sum_\pm \int_{\Gamma} \tau (\pitr^\pm u_h - \mu_h) \nu_h \ \mathrm{d} \sigma &= 0,
			\label{HDD:eq:dHDD:FEM:proj:3}
		\end{align}
	for all $(\bvec{w}_h,v_h,\nu_h)\in W^q(\mathcal{T}_h) \times Q^q(\mathcal{T}_h) \times M^q(\mathcal{F}_{\Gamma,h})$.
\end{subequations}
Here, $\pitr^\pm$ is a projection of the one-sided traces on $\Gamma$ onto $M^q(\cF_{\Gamma,h})$ that will be described later.
We consider the projection as the traces $\tr^\pm u_h, \tr^\pm v_h$ for functions $u_h, v_h \in Q^q(\cT_h)$. These are, in general, not polynomials after transformation onto the reference interval $[0,1]$, in difference to functions $\mu_h, \nu_h \in M^q(\cF_{\Gamma,h})$. %
This originates in the fact that $|\det \nabla \Phi_K|$ evaluated on one edge of the reference element $[0,1]^2$ is in general not a polynomial.
Note, that the evaluation of the integrals with a quadrature rule of exactness $2q$ on each edge of $\cF_{\Gamma,h}$ is equivalent to project the traces $\tr^\pm u_h, \tr^\pm v_h$ onto the space $M^q(\cF_{\Gamma,h})$.
\begin{theorem}[Well-posedness]
  \label{lem:HDD:FEM:Wellposedness}
  For any $f \in L^2(\OmG)$ there exist unique solutions $(\vq_h, u_h, \mu_h) \in W^q(\mathcal{T}_h) \times Q^q(\mathcal{T}_h) \times M^q(\mathcal{F}_{\Gamma,h})$ of~\eqref{HDD:eq:dHDD:FEM:proj} where it holds with a constant $C$ independent of $\tau$ that
  \begin{subequations}
  \label{eq:HDD:FEM:Wellposedness}
  \begin{align}
     \big\|\vq_h\big\|_{H(\Div,\Omega)}
     + \big\|u_h\big\|_{L^2(\OmG)}
     + \big\|\mu_h\big \|_{L^2(\Gamma)}
     &\leq C \big\|f \big\|_{L^2(\OmG)}, \\
	\big\|\jump{\vq_h\cdot\bvec{n}}\big\|_{L^2(\Gamma)} &\leq C \sqrt{\tau} \|f\|_{L^2(\Omega)},\label{eq:HDD:FEM:NormEstqn}\\
	  \sqrt{\tau}
	  \left(
	  \big\|\jump{\pitr u_h}\big\|_{L^2(\Gamma)} + \big\| \mean{\pitr u_h} - \mu_h\big\|_{L^2(\Gamma)}
	  \right)
	  &\leq C \|f\|_{L^2(\Omega)}\label{eq:HDD:FEM:NormEstjumpu} \\
	  \sum_\pm \tfrac{1}{\sqrt{1+\tau}} \big\| \sqrt{\tau} \pitr^\pm u_h \big\|_{L^2(\Gamma)} &\leq C \big\|f \big\|_{L^2(\OmG)},
  \end{align}
  \end{subequations}
\end{theorem}
For a fixed discretization, the jump of $u_h$ across $\Gamma$ decreases as $\tau$ increases, since $\mu_h$ approaches the average value of $u_h$ on $\Gamma$.
In contrast, the jump in $\vq\cdot\vn$ across $\Gamma$ decreases as $\tau$ decreases and vanishes exactly for $\tau = 0$. %
Moreover, when the mesh width is decreased or the order $q$ is increased, both the jump in $\pitr u_h$ -- and so in $\tr u_h$ -- and the discrepancy between $\mu_h$ and $\mean{\pitr u_h}$  -- and so $\mean{\tr u_h}$ -- on $\Gamma$ are reduced, independently of $\tau$. %
The observations are consistent with the findings of Cockburn et al., who argued that the choice $\tau\to 0$ leads to the \emph{hybrid mixed method} (or, equivalently, the \emph{Raviart-Thomas hybrid method}) \cite{Cockburn:2016, Cockburn:2023, Cockburn.Gopalakrishnan.Lazarov:2009}, whereas the limit $\tau\to\infty$ yields \emph{CG methods} \cite{Cockburn:2023, Cockburn.Gopalakrishnan.Lazarov:2009}. It was further observed that selecting $0<\tau<\infty$ is essential in order to permit hanging nodes, and thus the use of non-conforming meshes \cite{Cockburn.Gopalakrishnan.Lazarov:2009}.
Finally, we give estimates on the discretization error.
\begin{theorem}[Discretization error]
  Let $\mathcal{T}_h$ be quasi-uniform and shape-regular, let $(\vq_h, u_h, \mu_h) \in W^q(\mathcal{T}_h) \times Q^q(\mathcal{T}_h) \times M^q(\mathcal{F}_{\Gamma,h})$, $q\in\mathbb{N}_0$, be the solution of~\eqref{HDD:eq:dHDD:FEM:proj}
  and let $(\vq, u, \mu) \in \cW_\tau \times \cQ_{\tau/(1+\tau)} \times L^2(\Gamma)$ be the solution of \eqref{HDD:eq:var} with $u \in H^s(\OmG)$, $s > \frac12$.
  Then, there exist two constants $C_1, C_2 > 0$ and a function
		\begin{align}
			\operatorname{err}(h,q,\tau,u,C_1,C_2) :=
			\left(C_1 h^{\min(q+1,s)}
			 + C_2
			\tfrac{\sqrt{\tau}}{\sqrt{1+\tau}} h^{\min(q+\nicefrac12,s-\nicefrac12)}\right)
			\|u\|_{H^s(\OmG)}
			\label{eq:FEM:error_function}
		\end{align}
		such that
		\begin{subequations}
		    \label{eq:FEM:discretisation_error}
			\begin{align}
			    \| \vq - \vq_h \|_{H(\Div,\OmG)} + \| u - u_h\|_{L^2(\OmG)} + \|\mu - \mu_h\|_{L^2(\OmG)}
				&\leq
				\sqrt{1+\tau}\operatorname{err}(h,q,\tau, u, C_1, C_2),
				\label{eq:FEM:discretisation_error:uqmu} \\
				\big\|\jump{\vq_h\cdot\bvec{n}}\big\|_{L^2(\Gamma)}
				&\leq \hspace{2.3em} \tau \operatorname{err}(h,q,\tau, u, C_1, C_2),
				\label{eq:FEM:discretisation_error:qn}\\
				\sqrt{\tau}
				\big(
				\big\|\jump{\pitr u_h}\big\|_{L^2(\Gamma)} + \big\| \mean{\pitr u_h} - \mu_h\big\|_{L^2(\Gamma)}
				\big)
				&\leq (1+\tau)\operatorname{err}(h,q,\tau, u, C_1, C_2).
				\label{eq:FEM:discretisation_error:tr_u}
			\end{align}
		\end{subequations}
   \label{lem:FEM:discretisation_error}
\end{theorem}

\section{Function spaces and variational formulation}
\label{sec:HDD}
In this section, we introduce a variational formulation that serves as the basis for deriving and analyzing a perturbed Galerkin formulation in Section~\ref{sec:Galerkin}. We will show that the finite element formulation~\eqref{HDD:eq:dHDD:FEM:proj} can be regarded as conforming. We start by explaining the choice for the spaces and norms that are weighted with $\tau$ such that the norm of the solution is independent of $\tau$.
With these choices, the continuity constant of the bilinear form is increasing with~$\tau$ which becomes evident in the error analysis in Sec~\ref{sec:Galerkin}.

\subsection{Function spaces}

For the variable $u$ we desire as function space a subspace of $L^2(\OmG)$ with one-sided traces $\tr^\pm u$ in $L^2(\Gamma)$ for $\tau > 0$ such that the integrals $\int_\Gamma \tr^\pm u \tr^\pm v\,\text{d}\sigma$ in~\eqref{HDD:eq:HDD2} are well-defined. %
One choice might be $H^{\nicefrac12}(\OmG)$, for which then $\Div\vq$ shall be seeked in $H^{-\nicefrac12}(\OmG)$. %
Instead, we prefer to write a formulation for separate variables $u^\circ \in L^2(\Omega)$ and $\sqrt{\tau} u^\pm \in L^2(\Gamma)$, more precisely, we seek $u$ in some space that allows for a unique decomposition $u = u^\circ + \sum_{\pm} \sqrt{\tau} u^\pm$. %
Note, that in general, the space $L^2(\Omega)$ does not possess traces in $L^2(\Gamma)$, and even if a particular $u^\circ$ would possess traces, then it might be not equal to $u^\pm$.
A similar space $L^2(\Omega; \text{d}\vx + \text{d}\sigma)$ has been introduced in~\cite{terElst.Meyries.Rehberg:2014} through a natural topological isomorphism to the direct sum $L^2(\OmG) \oplus L^2(\Gamma)$, where the mean value of the one-sided traces on $\Gamma$ is considered. There, each $f_\Omega \in L^2(\Omega; \text{d}\vx + \text{d}\sigma)$ is a sum $f_\Omega = f_{\OmG} + f_\Gamma$ of a function $f_\OmG \in L^2(\OmG)$ and a distribution $f_\Gamma \in L^2(\Gamma)$, where the decomposition is unique. %
In difference to~\cite{terElst.Meyries.Rehberg:2014} we aim to have two possibly different functions on $\Gamma$, where the formulation may enforce them to be equal. If we would take the completion of the space of smooth functions $v \in C^\infty(\OmG)$ with respect to the norm defined by $\|v\|^2_{L^2(\OmG)} + \sum\nolimits_\pm \|\sqrt{\tau}\tr^\pm v\|_{L^2(\Gamma)}^2$, then each $v$ in this space can be written as the sum $v = v^\circ + \sum_\pm \sqrt{\tau} v^\pm$, where $v^\circ \in L^2(\OmG)$ and $\sqrt{\tau} v^\pm \in L^2(\Gamma)$. However, the decomposition is not unique and the space is only topological isomorph to $L^2(\OmG) \oplus L^2(\Gamma)$ for $\tau > 0$. %
Therefore, $v = \sqrt{\tau}v^+$ and $v = \sqrt{\tau}v^-$ can not be distinguished if $v^- = v^+$. %
To obtain a space with a unique decomposition $v = v^\circ + \sum_\pm \sqrt{\tau} v^\pm$ we aim to place the components $\sqrt{\tau}v^\pm$ on shifted interfaces $\Gamma^\pm$. For this we define the bijective mappings $\jmath^\pm: \Gamma \to \Gamma^\pm, \vx \mapsto \vx \pm \vd(\vx)$ for some vector field $\vd \in \Gamma$ such that $\Gamma^\pm \subset \OmG$, $\Gamma^\pm$ coincide only in $(d-1)$-dimensional null sets and
$c \|w\|_{L^2(\Gamma)} \leq \|w \circ (\jmath^\pm)^{-1}\|_{L^2(\Gamma^\pm)} \leq C \|w\|_{L^2(\Gamma)}$ for any $w \in L^2(\Gamma)$ and some $c,C > 0$. %
Now, we define the space $\cQ_{\tau/(1+\tau)}$, $\tau \geq 0$, as the completion of the space of smooth functions $v \in C^\infty_0(\OmG)$ with respect to the norm defined by
$\|v\|^2_{L^2(\OmG)} + \sum\nolimits_\pm \|\sqrt{\tau}\tr^\pm v \circ (\jmath^\pm)^{-1}\|_{L^2(\Gamma^\pm)}^2$. %
Following the arguments in~\cite{terElst.Meyries.Rehberg:2014}, we see that the space $\cQ_{\tau/(1+\tau)}$ is topological isomorph to the space $L^2(\Omega) \oplus L^2(\Gamma^+) \oplus L^2(\Gamma^-)$ for $\tau > 0$. %
This allows us to write each $v \in \cQ_{\tau/(1+\tau)}$ as the sum
\begin{align}
 \label{HDD:eq:Qtau:decomposition}
 v = v^\circ + \sum_\pm \sqrt{\tau} v^\pm \circ (\jmath^\pm)^{-1},
\end{align}
where $v ^\circ \in L^2(\OmG)$ and $\sqrt{\tau} v^\pm \in L^2(\Gamma)$ are uniquely defined. Moreover, for any $v \in \cQ_0$ it holds $v = v^\circ \in L^2(\OmG)$, \ie, the distributions on $\Gamma$ are zero and the two components $v^\pm$ do not appear anymore, and so $\cQ_0 = L^2(\OmG)$. %
We equip the space $\cQ_{\tau/(1+\tau)}$, $\tau \geq 0$ with a $\tau$-weighted norm defined by
\begin{align}
  \label{HDD:eq:Qtau:norm}
  \|v\|_{\cQ_{\tau/(1+\tau)}}^2 &:= \|v^\circ\|^2_{L^2(\OmG)} + \sum\nolimits_\pm \tfrac{1}{1+\tau} \|\sqrt{\tau} v^\pm\|_{L^2(\Gamma)}^2\ .
\end{align}
The weight is chosen such that the second term in the norm behaves like $\|\sqrt{\tau} v^\pm\|_{L^2(\Gamma)}$ for small $\tau$ and for large $\tau$ like $\|v^\pm\|_{L^2(\Gamma)}$. %
We will bound the solution in this norm with a constant independent of $\tau$.
For the hybrid variable $\mu$ and the corresponding test functions $\nu$ we choose the space $L^2(\Gamma)$ equipped with the standard norm $\|\cdot\|_{L^2(\Gamma)}$ such that $\int_{\Gamma}\tau\mu\nu\, \mathrm{d}\sigma$ in \eqref{HDD:eq:HDD3} is well-defined. %
The other integrals in \eqref{HDD:eq:HDD3} are then well-defined due to $\sqrt{\tau} u^\pm~\in~L^2(\Gamma)$ and if we require $\jump{\bvec{q}\cdot\bvec{n}}\in L^2(\Gamma)$.
Analogously to $Q_{\tau/(1+\tau)}$, we define the space $Q_\tau$, $\tau > 0$, where the weight $\nicefrac{1}{1+\tau}$ in~\eqref{HDD:eq:Qtau:norm} is replaced by $1$. Naturally, we define the duality product
\begin{align}
   \label{HDD:eq:Qtau':DP}
   \left<u, v\right>_{\cQ_\tau} :=
   \left<u^\circ, v^\circ\right> + \sum\nolimits_{\pm} \left<\sqrt{\tau} u^\pm, \sqrt{\tau} v^\pm\right>,
   \quad \tau \geq 0,
\end{align}
denoting for convenience by $\left<\cdot,\cdot\right>$ the $L^2(\OmG)$ or $L^2(\Gamma)$ duality product, respectively. %
Due to the topological isomorphism
\begin{align*}
  \cQ_\tau \cong L^2(\OmG) \oplus L^2(\Gamma^+) \oplus L^2(\Gamma^-) = (L^2(\OmG) \oplus L^2(\Gamma^+) \oplus L^2(\Gamma^-))^\prime
\end{align*}
it holds for the dual space that $\cQ_\tau^\prime \cong \cQ_\tau$, $\tau > 0$.
For $\tau = 0$ we define $\cQ_0' := \cQ_0 = L^2(\OmG)$.
The subspace of $\cQ_\tau$ with homogeneous distributions on $\Gamma$ is $L^2(\OmG)$.
In particular, we aim for a formulation, where $\Div \vq \in \cQ_\tau$ can be decomposed as
\begin{align}
  \label{HDD:eq:divvq:decomp}
  \Div \vq = -f + \sum\nolimits_{\pm} \sqrt{\tau} (u^\pm  - \mu) \circ (\jmath^\pm)^{-1},
\end{align}
with $f \in L^2(\OmG)$ and $\sqrt{\tau}(u^\pm - \mu) \in L^2(\Gamma)$.
Hence, the variable $\bvec{q}$ shall be seeked in the space
\begin{subequations}
\begin{equation}
	\cW_\tau\coloneq \left\lbrace\bvec{w} \in (L^2(\OmG))^2: \Div \vw \in \cQ_\tau, \ \tfrac{1}{\sqrt{1 + \tau}}\jump{\bvec{w}\cdot\bvec{n}}\in L^2(\Gamma)\right\rbrace\label{HDD:eq:Wtau}
\end{equation}
for $\tau\geq0$ with the norm defined by
\begin{equation}
	\|\bvec{w}\|^2_{\cW_\tau} \coloneq \|\bvec{w}\|^2_{L^2(\OmG)} + \|\Div\bvec{w}\|^2_{\cQ_\tau} + \tfrac{1}{1+\tau}\|\jump{\bvec{w}\cdot\bvec{n}}\|^2_{L^2(\Gamma)} .\label{HDD:eq:WtauNorm}
\end{equation}
\end{subequations}
Note, that $\cW_\tau \subset H(\Div, \OmG)$ for all $\tau \geq 0$.
On the product space $\cW_\tau \times \cQ_{\tau/(1+\tau)} \times L^2(\Gamma)$ we use the norm defined by
\begin{align*}
   \pnorm{(\vw, v, \nu)}^2
   := \|\vw\|_{\cW_\tau}^2 + \|v\|_{\cQ_{\tau/(1+\tau)}}^2 + \|\nu\|_{L^2(\Gamma)}^2\ .
\end{align*}

\subsection{The variational formulation}
To obtain a variational formulation in the product space $\cW_\tau \times \cQ_{\tau/(1+\tau)} \times L^2(\Gamma)$ we first replace $\tr^\pm u$ and $\tr^\pm v$ in~\eqref{HDD:eq:HDD2} and~\eqref{HDD:eq:HDD3}
by $u^\pm$ and $v^\pm$. %
Moreover, we assure for the decomposition~\eqref{HDD:eq:divvq:decomp} by replacing $\int_\OmG \Div \vq v \,\text{d}\vx$ in~\eqref{HDD:eq:HDD2} by $\left< \Div \vq, v\right>_{Q_\tau}$ and $\int_{\Gamma} \jumpt{\vq\cdot \vn} \nu \,\text{d}\sigma$ in~\eqref{HDD:eq:HDD3} by\\
$\left<\jumpt{\vq\cdot\vn} - \sum\nolimits_\pm \tau (\Div \vq)^\pm, \nu\right>$. %
This does not harm the consistency of the formulation if $\sqrt{\tau}(\Div \vq)^\pm = 0$.
However, this changes the structure of the formulation in a way that allows for an analysis with conforming subspaces.
Moreover, we replace $\int_\OmG u \Div \vw \,\text{d}\vx$ in~\eqref{HDD:eq:HDD1} by $\left< u, \Div \vw\right>_{Q_\tau} - \left<\mu, \sum_\pm \tau (\Div \vw)^\pm\right>$ which does not harm the consistency since $u^\pm = \mu$ on $\Gamma$.
Hence, we seek $(\vq, u, \mu) \in \cW_\tau \times \cQ_{\tau/(1+\tau)} \times L^2(\Gamma)$ such that
\begin{subequations}\label{HDD:eq:var}
\begin{align}
    \label{HDD:eq:var:1}
	\big<\tfrac{1}{\kappa}\vq, \vw\big> + \big< u, \Div \vw\big>_{\cQ_\tau} + \big<\mu, \jump{\vw\cdot\vn} - \sum\limits_\pm \tau (\Div \vw)^\pm \big> &= 0
	\quad \hspace{2.3em}\forall \vw \in \cW_\tau, \\
	\label{HDD:eq:var:2}
	-\big<\Div \vq, v\big>_{\cQ_\tau} +
	\sum\limits_\pm \tau \big< u^\pm - \mu, v^\pm \big> &= \big<f, v^\circ\big>
	\quad \forall v \in \cQ_{\tau/(1+\tau)},\\
	\label{HDD:eq:var:3}
	-\big<\jump{\vq\cdot\vn} - \sum\limits_\pm \tau (\Div \vq)^\pm, \nu\big> - \sum\nolimits_\pm \tau \big< u^\pm - \mu, \nu \big> &= 0
	\quad \hspace{2.4em}\forall \nu \in L^2(\Gamma).
\end{align}
\end{subequations}

\begin{remark}
	For $\tau=0$, the variational formulation~\eqref{HDD:eq:var} has saddle point structure, for which it is suggested to choose $H^{\nicefrac12}\hspace{-1em}{}_{00}(\Gamma)$ as the space for $\mu$, which takes the role of a Lagrange multiplier, and the corresponding test functions~$\nu$ (see~\cite{Arbogast.Pencheva.Wheeler.Yotov:2007,Brezzi.Fortin:1991}). %
	In this case, one may choose $H(\Div, \OmG)$ as space for $\bvec{q}$ and~$\bvec{w}$, for which the constraints $\jumpt{\bvec{q}\cdot\bvec{n}},\, \jumpt{\bvec{w}\cdot\bvec{n}}\in L^2(\Gamma)$ are dropped and $\jumpt{\bvec{q}\cdot\bvec{n}}\in H^{-\nicefrac12}(\Gamma)$.
	For the formulation with stabilization terms we seek the Lagrange multiplier in the weaker space $L^2(\Gamma)$ and the jumps $\jump{\bvec{q}\cdot\bvec{n}}$ in $L^2(\Gamma)$ instead of $H^{-\nicefrac12}(\Gamma)$. In addition, we have the $L^2(\Gamma)$ distributions $\sqrt{\tau}u^\pm$ and $\sqrt{\tau}(\Div \vq)^\pm$.
\end{remark}

\subsection{Well-Posedness}
We denote by $\blf{b}: (\cW_\tau \times \cQ_{\tau/(1+\tau)} \times L^2(\Gamma)) \times (\cW_\tau \times \cQ_{\tau/(1+\tau)} \times L^2(\Gamma))$ the bilinear form on the product space obtained as sum of the left hand sides of~\eqref{HDD:eq:var}, which is continuous. Hence,~\eqref{HDD:eq:var} is equivalent to:
Seek $(\vq, u, \mu) \in \cW_\tau \times \cQ_{\tau/(1+\tau)} \times L^2(\Gamma)$ such that
\begin{align}
  \label{eq:HDD:var:with_bf}
     \blf{b}((\vq, u, \mu), (\vw, v, \nu)) =
     \left<f, v^\circ\right>
     \quad \forall
     (\vw, v, \nu) \in \cW_\tau \times \cQ_{\tau/(1+\tau)} \times L^2(\Gamma)\ .
\end{align}
To show well-posedness we prove, first, the following auxiliary result.
\begin{lemma}
  \label{lem:HDD:deRham}
   For any $(v^\circ, \nu) \in L^2(\OmG) \times L^2(\Gamma)$
   there exists $\vw \in \cW_\tau$ with
   \begin{subequations}
   \label{eq:HDD:deRham:w}
   \begin{align}
   \label{eq:HDD:deRham:w:1}
   \jump{\vw\cdot\vn} &= \nu \text{ on } \Gamma & \text{ and } &&
    \Div \vw &= v^\circ\text{ in } \OmG\ ,
    \end{align}
    for which with a constant $C$ independent of $\tau$ it holds
    \begin{align}
       \label{eq:HDD:deRham:w:2}
       \|\vw\|_{H^{\nicefrac12}(\OmG)} \leq C \left( \| v^\circ\|_{L^2(\OmG)} + \| \nu \|_{L^2(\Gamma)}\right)\ .
    \end{align}
    \end{subequations}
\end{lemma}
\begin{proof}
  The statement is proven by constructing $\vw \in \cW_\tau$ with the above properties. %
  For this we let $z \in H^1(\OmG)$ be the solution of
  \begin{align}
     \label{eq:HDD:deRham:w:proof:var_for_z}
     \int_{\OmG} \nabla z \cdot \nabla z' \,\text{d}\vx + \int_{\Gamma} \jump{z} \jump{z'} \,\text{d}\sigma
     = -\int_\OmG v^\circ z' \,\text{d}\vx - \int_\Gamma \nu \mean{z'}\,\text{d}\sigma
  \end{align}
  for all $z' \in H^1(\OmG)$, which is unique by the Lax-Milgram lemma and satisfies the bound
  \begin{align}
     \label{eq:HDD:deRham:w:proof:1}
     \|z\|_{H^1(\OmG)} \leq C \left(\| v^\circ \|_{L^2(\OmG)} + \| \nu\|_{L^2(\Gamma)} \right).
  \end{align}
  Since $z$ solves the Poisson problem
  \begin{align*}
     \Delta z &= v^\circ \in L^2(\OmG), &
     \jump{\nabla z\cdot \vn} &= \nu \in L^2(\Gamma), &
     \mean{\nabla z\cdot \vn} &= \jump{z} \in H^{\nicefrac12}(\Gamma)
  \end{align*}
  with homogeneous Neumann boundary conditions $\nabla z \cdot \vn = 0$ on $\partial\Omega$,
  the conditions~\eqref{eq:HDD:deRham:w:1} are satisfied by $\vw = \nabla z$. Since the two-sided Neumann traces $\nabla z \cdot \vn$
  are both in $L^2(\Gamma)$ and $\nabla z \cdot \vn = 0$ we can assert that $z \in H^{\nicefrac32}(\OmG)$, which implies~\eqref{eq:HDD:deRham:w:2}, and the proof is complete.
\end{proof}

\begin{proposition}[inf-sup conditions]
  \label{lem:HDD:inf-sup}
  The bilinear form $\blf{b}$ fulfills the inf-sup conditions
  \begin{subequations}
     \label{eq:HDD:inf-sup}
     \begin{align}
       \label{eq:HDD:inf-sup:1}
        \sup_{\substack{(\vw', v', \nu') \in \\\cW_\tau \times \cQ_{\tau/(1+\tau)} \times L^2(\Gamma) \\\setminus ( \zerobf, 0, 0)}}
        \frac{\blf{b}((\vw, v, \nu), (\vw', v', \nu'))} {\pnorm{(\vw', v', \nu')}}
        &\geq \gamma
        \pnorm{(\vw, v, \nu)}
        \quad \\[-1.5em]
        \nonumber
        \hspace{2em}\forall
        (\vw, v, \nu) &\in \cW_\tau \times \cQ_{\tau/(1+\tau)} \times L^2(\Gamma)\setminus ( \zerobf, 0, 0)
        \\[0.5em]
        \label{eq:HDD:inf-sup:2}
        \sup_{\substack{(\vw, v, \nu) \in \\ \cW_\tau \times \cQ_{\tau/(1+\tau)} \times L^2(\Gamma)}}
        \blf{b}((\vw, v, \nu), (\vw', v', \nu'))
        &> 0 \quad \\[-1.5em]
        \forall
        (\vw', v', \nu') &\in \cW_\tau \times \cQ_{\tau/(1+\tau)} \times L^2(\Gamma) \setminus ( \zerobf, 0, 0)
        \nonumber
     \end{align}
  \end{subequations}
  for some $\gamma > 0$ independent of $\tau$.
\end{proposition}
\begin{proof}
  We start to show~\eqref{eq:HDD:inf-sup:2} by considering $(\vw', v', \nu') \in \cW_\tau \times \cQ_{\tau/(1+\tau)} \times L^2(\Gamma) \setminus (\zerobf, 0, 0)$ such that
  \begin{align*}
    0 = \blf{b} ( (\vw, v, \nu), (\vw', v', \nu')) \quad
    \forall (\vw, v, \nu) \in \cW_\tau \times \cQ_{\tau/(1+\tau)} \times L^2(\Gamma)\ .
  \end{align*}
  First, choosing $(\vw, v, \nu) = (\vw', v', \nu')$ we find
  \begin{align*}
     0 = \|\tfrac{1}{\sqrt{\kappa}} \vw'\|_{L^2(\OmG)}^2 + \sum\nolimits_{\pm} \tau \| (v')^\pm - \nu'\|_{L^2(\Omega)}^2,
  \end{align*}
   and so $\vw' = \zerobf$ and $\sqrt{\tau}(v')^\pm = \sqrt{\tau}\nu'$.\\
   Now, choosing $(\vw, v, \nu) = (\vw, 0, 0)$, that satisfies $\jump{\vw\cdot\vn} = \nu'$, $\Div\vw = (v')^\circ$ which is possible by Lemma~\ref{lem:HDD:deRham}, we find that
   \begin{align*}
      0 = -\|(v')^\circ\|_{L^2(\OmG)}^2 - \|\nu'\|_{L^2(\OmG)}^2
   \end{align*}
   and so $(v')^\circ = 0$ and $\nu' = 0$. The latter implies $\sqrt{\tau}(v')^\pm = 0$, hence, $(\vw', v', \nu') = (\zerobf, 0, 0)$. This is a contradiction and so~\eqref{eq:HDD:inf-sup:2} holds.
  To show~\eqref{eq:HDD:inf-sup:1}, we prove that any $(\vw, v, \nu) \in \cW_\tau \times \cQ_{\tau/(1+\tau)} \times L^2(\Gamma)$ solving
  \begin{multline}
    \label{eq:HDD:inf-sup:proof:1}
    \blf{b} ((\vw, v, \nu), (\vw', v', \nu')) =
     \left< \vf_\vq, \vw\right>_{\cW_\tau} +
     \left<f_u, v\right>_{\cQ_{\tau/(1+\tau)}} +
     \left<f_\mu, \nu\right> \quad \\
     \forall (\vw', v', \nu')\in \cW_\tau \times \cQ_{\tau/(1+\tau)} \times L^2(\Gamma)
  \end{multline}
  satisfies with a constant $C$ independent of $\tau$ that
  \begin{align}
    \label{eq:HDD:inf-sup:proof:2}
      \|\vq\|_{\cW_\tau} + \|u\|_{\cQ_{\tau/(1+\tau)}} + \|\mu\|_{L^2(\Gamma)}
      \leq C \left( \|\vf_\vq\|_{\cW_\tau'} + \|f_u\|_{\cQ_{\tau/(1+\tau)|}} + \|f_\mu\|_{L^2(\Gamma)}\right).
   \end{align}
  We have divided the proof of~\eqref{eq:HDD:inf-sup:proof:2} into five steps.
  \smallskip
  \begin{enumerate}
   \item By Lemma~\ref{lem:HDD:deRham} we can find $\vw' \in \cW_\tau$ for which
  \begin{align}
    \jump{\vw'\cdot\vn} &= \nu \text{ on } \Gamma & \text{ and } &&
    \Div \vw' &= v^\circ\text{ in } \OmG.
  \end{align}
  Using $(\vw',0,0)$ as test function in~\eqref{eq:HDD:inf-sup:proof:1} we find
  \begin{align}
     \nonumber
     \| v^\circ \|_{L^2(\OmG)}^2 + \| \nu\|_{L^2(\Gamma)}^2
     &= -\left< \tfrac{1}{\kappa} \vw, \vw'\right> + \|\vf_\vq\|_{\cW_\tau'} \|\vw'\|_{\cW_\tau}\\
     &\leq C \left(\|\vw\|_{L^2(\OmG)} \|\vw' \|_{L^2(\OmG)} + \|\vf_\vq\|_{\cW_\tau'} \|\vw'\|_{\cW_\tau}\right).
     \label{eq:inf-sup:proof:u0_mu_by_q_fq}
  \end{align}
   and so using~\eqref{eq:HDD:deRham:w:2}
   \begin{align}
     \label{eq:HDD:inf-sup:proof:u0_mu_by_q}
      \| v^\circ \|_{L^2(\OmG)} + \| \nu\|_{L^2(\Gamma)} \leq C \|\vw\|_{L^2(\OmG)}
      + \| \vf_\vq \|_{\cW_\tau'} \ .
   \end{align}
  \item Now, testing \eqref{eq:HDD:inf-sup:proof:1} with $(\vq, u, \mu)$
  we find that
  \begin{multline}
	\tfrac{1}{\kappa_\mathrm{sup}}\| \vw\|_{L^2(\OmG)}^2 + \sum\nolimits_\pm \tau \| v^\pm - \nu \|_{L^2(\Gamma)}^2\\
	\leq \|\vf_\vq\|_{\cW_\tau'} \|\vw\|_{\cW_\tau}
	+ \| f_u \|_{\cQ_{\tau/(1+\tau)}} \| v \|_{\cQ_{\tau/(1+\tau)}}
	+ \| f_\mu \|_{L^2(\Gamma)} \| \nu \|_{L^2(\Gamma)}
  \end{multline}
  \item Since $\Div\vw \in \cQ_\tau = \cQ_\tau'$ we can assert that
  \begin{align*}
     \|\Div \vw\|_{\cQ_\tau} %
     &= \sup_{v'\in \cQ_{\tau} \setminus \{ 0 \}} \frac{\left<\Div \vw, v'\right>}{\|v'\|_{\cQ_\tau}}\ .
   \end{align*}
  Now, inserting~\eqref{eq:HDD:inf-sup:proof:1} for $(\vw', v', \nu') = (\zerobf, v', 0)$ we find
  \begin{align}
    \begin{aligned}
     \|\Div \vw\|_{\cQ_\tau} &=
     \sup_{v' \in \cQ_\tau \setminus \{ 0 \}} \frac{ \left<-f_u, v'\right>_{\cQ_{\tau}} + \sum_\pm \tau\left<(v^\pm - \nu), (v')^\pm\right>}{\sqrt{\|(v')^\circ\|_{L^2(\OmG)}^2 + \sum_\pm \|\sqrt{\tau}(v')^\pm\|_{L^2(\Gamma)}^2}}\\
     &\leq \|f_u\|_{\cQ_\tau} + \sum\nolimits_{\pm} \sqrt{\tau}\|v^\pm - \nu\|_{L^2(\Gamma)}
     \ .
     \end{aligned}
     \label{eq:HDD:inf-sup:proof:divq}
  \end{align}
  \item As $\jump{\cW_\tau\cdot \vn} \subseteq L^2(\Gamma)$ by assumption we can choose
  $(\zerobf, 0, -\tfrac{1}{\sqrt{1+\tau}}\jump{\vw\cdot \vn})$ as test function in~\eqref{eq:HDD:inf-sup:proof:1} and obtain
  with~\eqref{eq:HDD:inf-sup:proof:divq}
  \begin{align}
     \nonumber
     \tfrac{1}{\sqrt{1+\tau}}
     \| \jump{\vw\cdot \vn}\|_{L^2(\Gamma)}
     &\leq \tfrac{1}{\sqrt{1+\tau}} \left( \tau \| v^\pm - \nu \|_{L^2(\Gamma)} + \|f_u\|_{\cQ_\tau} \right) \\
     &\leq \sqrt{\tau} \| v^\pm - \nu \|_{L^2(\Gamma)} + \|f_u\|_{\cQ_\tau}.
  \end{align}
  \item Using the triangle inequality we find
  \begin{align}
     \label{eq:HDD:inf-sup:proof:upm_by_upm-mu_mu}
	    \tfrac{\sqrt{\tau}}{\sqrt{1+\tau}} \| v^\pm\|_{L^2(\Gamma)}
	    &\leq \tfrac{\sqrt{\tau}}{\sqrt{1+\tau}} \| v^\pm - \nu\|_{L^2(\Gamma)}
	    + \tfrac{\sqrt{\tau}}{\sqrt{1+\tau}}  \| \nu\|_{L^2(\Gamma)}\ .
	\end{align}
	\end{enumerate}
   Combining~\eqref{eq:HDD:inf-sup:proof:u0_mu_by_q}--\eqref{eq:HDD:inf-sup:proof:upm_by_upm-mu_mu} we find~\eqref{eq:HDD:inf-sup:proof:2}.
\end{proof}

\begin{theorem}[Well-posedness]
  \label{lem:HDD:Wellposedness}
  For any $f \in L^2(\OmG)$ there exists a unique solution $(\vq, u, \mu) \in \cW_\tau \times \cQ_{\tau/(1+\tau)} \times L^2(\Gamma)$ of~\eqref{HDD:eq:var} for which $(\Div \vq)^\pm = \jump{\vq\cdot\vn} = [u] = 0$ and $\mu = u^\pm$ on $\Gamma$, and it holds with a constant $C$ independent of $\tau$ that
  \begin{align}
    \label{eq:HDD:Wellposedness}
     \|\vq\|_{\cW_\tau} + \|u\|_{\cQ_{\tau/(1+\tau)}} + \|\mu\|_{L^2(\Gamma)} \leq C \|f\|_{L^2(\OmG)}.
  \end{align}
  Moreover, $(\vq, u, \mu)$ coincides with $(\kappa\Grad \mathrm{u}, \mathrm{u} + \sum_\pm \sqrt{\tau} \mathrm{u}|_\Gamma \circ (\jmath^\pm)^{-1}, \mathrm{u}|_\Gamma)\in \mathcal{W}_\tau \times \mathcal{Q}_{\tau/(1+\tau)} \times L^2(\Gamma)$, where $\mathrm{u} \in H^1_0(\Omega)$ is the solution of the primal formulation
  \begin{align}
    \label{eq:PrimalFormulation}
		\int_\Omega \kappa \Grad \mathrm{u} \cdot \Grad \mathrm{v} \, \mathrm{d}\bvec{x} = \int_\Omega f\mathrm{v} \,\mathrm{d}\bvec{x}\quad \forall \mathrm{v}\in H^1_0(\Omega).
  \end{align}
\end{theorem}
\begin{proof}
  The well-posedness of~\eqref{HDD:eq:var} and~\eqref{eq:HDD:Wellposedness} follows directly as the continuous bilinear form $\blf{b}$ fulfills the inf-sup conditions~\eqref{eq:HDD:inf-sup}.
  Since $L^2(\Omega) = L^2(\OmG)$ and so $f \in L^2(\Omega)$ the Lax-Milgram lemma implies the well-posedness of~\eqref{eq:PrimalFormulation}.
  Now, we verify that $(\vq, u, \mu) = (\kappa\Grad \mathrm{u}, \mathrm{u} + \sum_\pm \sqrt{\tau} \mathrm{u}|_\Gamma \circ (\jmath^\pm)^{-1}, \mathrm{u}|_\Gamma)$ with $(\Div \vq)^\pm = 0$ on $\Gamma$ is a solution to~\eqref{HDD:eq:var}.
  By definition it holds $u^\pm = \mu = \mathrm{u}|_\Gamma$. And as for solution $\mathrm{u}$ of the primal formulation it holds $\jump{\kappa\Grad\mathrm{u}\cdot\vn} = 0$ on $\Gamma$ and so $\jump{\vq\cdot\vn} = 0$ and~\eqref{HDD:eq:var:3} is satisfied.
  As $(\Div \vq)^\pm = 0$ and $u^\pm = \mu$ it suffices to verify~\eqref{HDD:eq:var:2} for test functions $v \in L^2(\OmG)$ for which $v^\pm = 0$ on $\Gamma$. Since~\eqref{eq:PrimalFormulation} implies $-\Div\kappa\Grad\mathrm{u} = f \in L^2(\OmG)$ we can assert that $-\Div\vq = f \in L^2(\OmG)$ and so~\eqref{HDD:eq:var:2} for all $v \in L^2(\OmG)$.
  Finally, since $\sqrt{\tau}\mu = \sqrt{\tau}u^\pm$ it remains to show that
  \begin{align*}
     \left<\tfrac{1}{\kappa}\vq, \vw\right> + \left<u, \Div\vw\right> + \left<\mu, \jump{\vw\cdot\vn}\right> = 0 \quad \forall \vw \in \cW_\tau\ .
  \end{align*}
  With $\tfrac{1}{\kappa} \vq = \Grad u$ in $\OmG$ and $\mu = \tr^\pm u$ on $\Gamma$ the equation is satisfied
  by the patch-wise integration by parts formula. Hence, the above equation and so~\eqref{HDD:eq:var:1} is fulfilled.
  Hence, we have constructed the unique solution of the HDD formulation~\eqref{HDD:eq:var}, for which $(\Div \vq)^\pm = \jump{\vq\cdot\vn} = 0$ and $\mu = u^\pm$ on $\Gamma$ hold. This completes the proof.
\end{proof}

\section{Galerkin discretization}
\label{sec:Galerkin}
In this section we introduce a perturbed Galerkin formulation of the HDD formulation introduced in Sec.~\ref{sec:HDD}. %

\subsection{Function spaces for the Galerkin formulation}
In this subsection we aim to introduce well-chosen subspaces $\wideparen{\cW} \times \wideparen{\cQ} \times \wideparen{\cM} \subseteq \cW_\tau \times \cQ_{\tau/(1+\tau)} \times L^2(\Gamma)$ for the perturbed Galerkin formulation, which will allow us to analyze it in a conforming setting with a consistency error.
Here, we also may omit that $\wideparen{v}^\pm$ and $\sqrt{\tau}\wideparen{v}^\pm$ in $\wideparen{v} = \wideparen{v}^\circ + \sum_\pm\sqrt{\tau}\wideparen{v}^\pm \circ (\jmath^\pm)^{-1}\in \wideparen{\cQ}$ are independent variables. %
Especially, it will be interesting to consider the space $\wideparen{\cQ} = \wideparen{\cQ}^\circ + \sum_\pm\sqrt{\tau}\pitr^\pm \wideparen{\cQ}^\circ$, where $\wideparen{\cQ}^\circ$ is a subspace of $L^2(\OmG)$ with $L^2(\Gamma)$ traces on~$\Gamma$ that are projected onto $\wideparen{\cM}$.
The subspace $\wideparen{\cW}$ is chosen such that $(\Div\wideparen{\cW})^\circ \subseteq \wideparen{\cQ}^\circ$
with an additional condition on $(\Div \wideparen{\cW})^\pm$ on $\Gamma$ to obtain
\begin{align}
   \| \Div \wideparen{\vw} \|_{\cQ_\tau}
   = \sup_{v \in \cQ_\tau \setminus \{ 0 \}} \frac{\left<\Div \wideparen{\vw}, v\right>_{\cQ_\tau}}{\|v\|_{\cQ_\tau}} =
   \sup_{\wideparen{v} \in \wideparen{\cQ} \setminus \{ 0 \}} \frac{\left<\Div \wideparen{\vw}, \wideparen{v}\right>}{\|\wideparen{v}\|_{\cQ_\tau}} =: \|\Div \wideparen{\vw}\|_{\wideparen{Q}'} \ .
   \label{eq:wideparencQ:dual:norm}
\end{align}
In this way, the stability estimate on $\| \Div \wideparen{\vw} \|_{\cQ_\tau}$ in the inf-sup conditions can be proven in Proposition~\ref{lem:HDD:Galerkin:inf-sup} similarly to $\| \Div \vw \|_{\cQ_\tau}$ in the proof of Proposition~\ref{lem:HDD:inf-sup}.
For the equality~\eqref{eq:wideparencQ:dual:norm} to hold we cannot simply choose $\sqrt{\tau}(\Div \wideparen{\cW})^\pm = 0$ on $\Gamma$, where ~\eqref{eq:wideparencQ:dual:norm} could only hold with the $L^2(\OmG)$-norm of $\wideparen{v}$ in the denominator.
More precisely, we seek the divergence $\Div \wideparen{\vw}$ in the Banach space adjoint
\begin{align}
  \wideparen{\cQ}' := \left.\left\{ \varphi \in \cQ_\tau : \varphi^\circ \in \wideparen{\cQ}^\circ,
  \left<\varphi, v\right>_{\cQ_\tau} =
  \begin{cases}
     \left<\varphi, v\right>, & v \in \wideparen{\cQ}, \\
     0, & v \in \wideparen{\cQ}^\perp
  \end{cases}
  \right\}\right\}\ ,
  \label{eq:wideparencQ:dual}
\end{align}
which is defined using the orthogonal decomposition $\cQ_\tau = \wideparen{\cQ} \oplus \wideparen{\cQ}^\perp$.
In this way,
\begin{subequations}
\label{eq:wideparencQ:dual:divq_paren}
\begin{align}
   \sum_\pm \tau \big<(\Div \wideparen{\vw})^\pm, \wideparen{v}^\pm\big> &= 0 \qquad \hspace{6.2em}\forall \wideparen{v} \in \wideparen{Q}\ .
   \label{eq:wideparencQ:dual:1}
   \intertext{However, $\sqrt{\tau}(\Div \wideparen{\vq})^\pm \neq 0$ in general for $\tau > 0$, since}
   \sum_\pm \tau \big<(\Div \wideparen{\vw})^\pm, (\wideparen{v}^\bot)^\pm\big> &= -\big< (\Div \wideparen{\vq})^\circ, (\wideparen{v}^\bot)^\circ\big> \qquad \forall \wideparen{v}^\bot \in \wideparen{Q}^\perp\ ,
   \label{eq:wideparencQ:dual:2}
\end{align}
\end{subequations}
and the $L^2(\OmG)$ and $L^2(\Gamma)$ contributions are linked together.

\subsection{The Galerkin formulation}
Replacing the solution space $\cW_\tau \times \cQ_{\tau/(1+\tau)} \times L^2(\Gamma)$ in~\eqref{HDD:eq:var}
by $\wideparen{\cW} \times \wideparen{\cQ} \times \wideparen{\cM}$, we obtain a Galerkin formulation in which,
due to the definition of $\wideparen{\cQ}^\prime$ in~\eqref{eq:wideparencQ:dual}, the $\cQ_\tau$ duality products are equal to the corresponding $L^2(\OmG)$ duality products.
Since in general $\sqrt{\tau}\wideparen{\cM} \nsubseteq \sqrt{\tau}\wideparen{\cQ}^\pm$, the other divergence terms on $\Gamma$ would not vanish in a Galerkin formulation.
As in a finite element formulation those terms would be very impractical, we prefer to consider and analyze a perturbed Galerkin formulation, in which these terms are not present: Seek $(\wideparen{\vq}, \wideparen{u}, \wideparen{\mu}) \in \wideparen{\cW} \times \wideparen{\cQ} \times \wideparen{\cM}$ such that
\begin{subequations}\label{dHDD:eq:var}
\begin{align}
    \label{dHDD:eq:var:1}
	\left<\tfrac{1}{\kappa}\wideparen{\vq}, \wideparen{\vw}\right> + \left< \wideparen{u}, \Div \wideparen{\vw}\right> + \left<\wideparen{\mu}, \jump{\wideparen{\vw}\cdot\vn}\right> &= 0
	\quad \hspace{2.3em}\forall \wideparen{\vw} \in \wideparen{\cW}, \\[0.5em]
	\label{dHDD:eq:var:2}
	-\left<\Div \wideparen{\vq}, \wideparen{v}\right> +
	\sum\nolimits_\pm \tau \left< \wideparen{u}^\pm - \wideparen{\mu}, \wideparen{v}^\pm \right> &= \left<f, \wideparen{v}^\circ\right>
	\quad \forall \wideparen{v} \in \wideparen{\cQ},\\[0.5em]
	\label{dHDD:eq:var:3}
	-\left<\jump{\wideparen{\vq}\cdot\vn}, \wideparen{\nu}\right> - \sum\nolimits_\pm \tau \left< \wideparen{u}^\pm - \wideparen{\mu}, \wideparen{\nu} \right> &= 0
	\quad \hspace{2.4em}\forall \wideparen{\nu} \in \wideparen{\cM}.
\end{align}
\end{subequations}

\subsection{Well-posedness}
We denote by $\wideparen{\blf{b}}: (\wideparen{\cW} \times \wideparen{\cQ} \times \wideparen{\cM}) \times (\wideparen{\cW} \times \wideparen{\cQ} \times \wideparen{\cM})$ the bilinear form on the product space obtained as sum of the left hand sides of~\eqref{dHDD:eq:var}, which is again continuous.
Hence,~\eqref{dHDD:eq:var} is equivalent to: Seek $(\wideparen{\vq}, \wideparen{u}, \wideparen{\mu}) \in \wideparen{\cW} \times \wideparen{\cQ} \times \wideparen{\cM}$ such that
\begin{align}
  \label{HDD:Galerkin:general}
     \wideparen{\blf{b}}((\wideparen{\vq}, \wideparen{u}, \wideparen{\mu}), (\wideparen{\vw}, \wideparen{v}, \wideparen{\nu})) =
     \left<f, \wideparen{v}^\circ\right>
     \quad \forall
     (\wideparen{\vw}, \wideparen{v}, \wideparen{\nu}) \in \wideparen{\cW} \times \wideparen{\cQ} \times \wideparen{\cM}.
\end{align}

To show well-posedness we prove, first, the auxiliary result.
\begin{lemma}
  \label{lem:HDD:Galerkin:deRham}
   Let $\wideparen{\cW} \subseteq \cW_\tau$, $\wideparen{\cQ} \subseteq \cQ_{\tau/(1+\tau)}$, with  $\wideparen{\cQ}^\circ \subseteq (\Div \wideparen{\cW})^\circ$ and $\wideparen{\cM} = \jump{\wideparen{\cW} \cdot \vn}$, let there be a projection $\wideparen{\Pi}$ of $(H^{\nicefrac12}(\OmG))^d \cap H(\Div,\OmG)$ onto $\wideparen{\cW}$ such that
   \begin{subequations}
   \begin{align}
      \left< \Div(\vw - \wideparen{\Pi} \vw)^\circ, (\wideparen{v}')^\circ\right> &= 0 \quad \forall (\wideparen{v}')^\circ \in \wideparen{\cQ}^\circ,\\
      \left< \jump{(\vw - \wideparen{\Pi} \vw)\cdot \vn}, \jump{\wideparen{\vw}'\cdot\vn}\right> &= 0 \quad \forall \wideparen{\vw}' \in \wideparen{\cW},
   \end{align}
   \end{subequations}
   with $\|\wideparen{\Pi} \vw\|_{L^2(\OmG)} \leq C \left(\| \vw \|_{H^{\nicefrac12}(\OmG)} + \| (\Div \vw)^\circ \|_{L^2(\OmG)}\right)$.
   Then, for any $(\wideparen{v}^\circ, \wideparen{\nu}) \in \wideparen{\cQ}^\circ \times \wideparen{\cM}$
   there exists $\wideparen{\vw} \in \wideparen{\cW}$ with
   \begin{subequations}
   \label{eq:HDD:deRham:wparen}
   \begin{align}
   \label{eq:HDD:deRham:wparen:1}
   \jump{\wideparen{\vw}\cdot\vn} &= \wideparen{\nu} \text{ on } \Gamma & \text{ and } &&
    (\Div \wideparen{\vw})^\circ &= \wideparen{v}^\circ\text{ in } \OmG\ ,
    \end{align}
    for which with a constant $C$ independent of $\tau$ it holds
    \begin{align}
       \label{eq:HDD:deRham:wparen:2}
       \|\wideparen{\vw}\|_{L^2(\OmG)} \leq C \left( \| \wideparen{v}^\circ\|_{L^2(\OmG)} + \| \wideparen{\nu} \|_{L^2(\Gamma)}\right)\ .
    \end{align}
    \end{subequations}
\end{lemma}
\begin{proof}
  First, we note that for any $\wideparen{\vw}$ fulfilling~\eqref{eq:HDD:deRham:wparen:1} the projection $\wideparen{\Pi}\wideparen{\vw}$ fulfills $\jump{\wideparen{\Pi}\wideparen{\vw}\cdot\vn} = \jump{\wideparen{\vw}\cdot\vn} = \wideparen{\nu}$ and $\Div \wideparen{\Pi}\wideparen{\vw} =
  \Div \wideparen{\vw} = \wideparen{v}^\circ$.
  The statement is proven for a particular $\wideparen{\vw}$. For this we let $z \in H^1(\OmG)$ be the
  unique solution of~\eqref{eq:HDD:deRham:w:proof:var_for_z}, that fulfills
  \begin{align}
     \|z\|_{H^{\nicefrac32}(\OmG)} \leq C \left(\| \wideparen{v}^\circ \|_{L^2(\OmG)} + \| \wideparen{\nu}\|_{L^2(\Gamma)} \right)\ .
  \end{align}
  and $\wideparen{\vw}= \wideparen{\Pi} \nabla z$ fulfills the conditions~\eqref{eq:HDD:deRham:wparen:1}.
  Using the bound on the projection operator $\wideparen{\Pi}$, we find that
  \begin{align}
    \begin{aligned}
     \|\wideparen{\vw}\|_{L^2(\OmG)} &\leq C \big(\| \nabla z \|_{H^{\nicefrac12}(\OmG)} + \|\Delta z\|_{L^2(\OmG)}\big) \\
     &\leq C \left(\| \wideparen{v}^\circ \|_{L^2(\OmG)} + \| \wideparen{\nu}\|_{L^2(\Gamma)} \right)\ ,
     \end{aligned}
  \end{align}
  and so~\eqref{eq:HDD:deRham:wparen:2}, which completes the proof.
\end{proof}

\begin{proposition}[inf-sup conditions] \label{lem:HDD:Galerkin:inf-sup}
  Let the assumptions of Lemma~\ref{lem:HDD:Galerkin:deRham} hold and let $\Div \wideparen{\cW}\subseteq \wideparen{\cQ}'$. Then, the bilinear form $\wideparen{\blf{b}}$ fulfills the inf-sup conditions
  \begin{subequations}
     \label{eq:HDD:Galerkin:inf-sup}
     \begin{align}
       \label{eq:HDD:Galerkin:inf-sup:1}
        \sup_{\substack{(\wideparen{\vw}', \wideparen{v}', \wideparen{\nu}') \in \\\wideparen{\cW} \times \wideparen{\cQ} \times \wideparen{\cM} \\\setminus \{ \zerobf, 0, 0)}}
        \frac{\wideparen{\blf{b}}((\wideparen{\vw}, \wideparen{v}, \wideparen{\nu}), (\wideparen{\vw}', \wideparen{v}', \wideparen{\nu}'))} {\|(\wideparen{\vw}', \wideparen{v}', \wideparen{\nu}'\|_{\wideparen{\cW} \times \wideparen{\cQ} \times \wideparen{\cM}}}
        &\geq \gamma
        \|(\wideparen{\vw}, \wideparen{v}, \wideparen{\nu}'\|_{\wideparen{\cW} \times \wideparen{\cQ} \times \wideparen{\cM}}
        \quad \\[-2em]
        \nonumber
        &\hspace{2em}\forall
        (\wideparen{\vw}, \wideparen{v}, \wideparen{\nu}) \in \wideparen{\cW} \times \wideparen{\cQ} \times \wideparen{\cM}\setminus \{ \zerobf, 0, 0)
        \\[0.5em]
        \label{eq:HDD:Galerkin:inf-sup:2}
        \sup_{\substack{(\wideparen{\vw}, \wideparen{v}, \wideparen{\nu}) \in \\ \wideparen{\cW} \times \wideparen{\cQ} \times \wideparen{\cM}}}
        \wideparen{\blf{b}}((\wideparen{\vw}, \wideparen{v}, \wideparen{\nu}), (\wideparen{\vw}', \wideparen{v}', \wideparen{\nu}'))
        &> 0 \quad \forall
        (\wideparen{\vw}', \wideparen{v}', \wideparen{\nu}') \in \wideparen{\cW} \times \wideparen{\cQ} \times \wideparen{\cM}\setminus \{ \zerobf, 0, 0)
     \end{align}
  \end{subequations}
  for some $\gamma > 0$ independent of $\tau$ that depends on the spaces only through the constant in the bound of the projection operator~$\wideparen{\Pi}$.
\end{proposition}
\begin{proof}
  The proof of~\eqref{eq:HDD:Galerkin:inf-sup:2} is in analogy to the proof of~\eqref{eq:HDD:inf-sup:2},
  where Lemma~\ref{lem:HDD:Galerkin:deRham} is used instead of Lemma~\ref{lem:HDD:deRham}
  and~\eqref{eq:HDD:deRham:wparen:2} instead of~\eqref{eq:HDD:deRham:w:2}.
  Also~\eqref{eq:HDD:Galerkin:inf-sup:1} is shown similarly to~\eqref{eq:HDD:inf-sup:1}  by proving that any $(\wideparen{\vw}, \wideparen{v}, \wideparen{\nu}) \in \wideparen{\cW} \times \wideparen{\cQ} \times \wideparen{\cM}$ solving
  \begin{multline}
    \label{eq:HDD:Galerkin:inf-sup:proof:1}
    \wideparen{\blf{b}} ((\wideparen{\vw}, \wideparen{v}, \wideparen{\nu}), (\wideparen{\vw}', \wideparen{v}', \wideparen{\nu}')) =
     \left< \vf_\vq, \vw\right>_{\cW_\tau} +
     \left<f_u, \wideparen{v}\right>_{\cQ_{\tau/(1+\tau)}} +
     \left<f_\mu, \wideparen{\nu}\right> \quad \\
     \forall (\wideparen{\vw}', \wideparen{v}', \wideparen{\nu}')\in \wideparen{\cW} \times \wideparen{\cQ} \times \wideparen{\cM}
  \end{multline}
  satisfies with a constant $C$ independent of $\tau$ that
  \begin{align}
    \label{eq:HDD:Galerkin:inf-sup:proof:2}
      \|\wideparen{\vq}\|_{\cW_\tau} + \|\wideparen{u}\|_{\cQ_{\tau/(1+\tau)}} + \|\wideparen{\mu}\|_{L^2(\Gamma)}
      \leq C \left( \|\vf_\vq\|_{\cW_\tau'} + \|f_u\|_{\cQ_{\tau/(1+\tau)|}} + \|f_\mu\|_{L^2(\Gamma)}\right).
   \end{align}
   Here, all steps follow similarly to the proof of~\eqref{eq:HDD:inf-sup:proof:2}, and even the bound on $\Div\wideparen{\vw} \in \wideparen{\cQ}'$ is obtained in a similar manner as in the proof for the variational formulation.
   More precisely, inserting~\eqref{eq:HDD:Galerkin:inf-sup:proof:1} for $(\wideparen{\vw}', \wideparen{v}', \wideparen{\nu}') = (\zerobf, \wideparen{v}', 0)$ and using~\eqref{eq:wideparencQ:dual:norm}
   we find
  \begin{align}
    \begin{aligned}
     \|\Div \wideparen{\vw}\|_{\cQ_\tau} &=
     \sup_{\wideparen{v}' \in \wideparen{\cQ} \setminus \{ 0 \}} \frac{ \left<-f_u, \wideparen{v}'\right>_{\cQ_{\tau}} + \sum_\pm \tau\left<(\wideparen{v}^\pm - \wideparen{\nu}), (\wideparen{v}')^\pm\right>}{\sqrt{\|(\wideparen{v}')^\circ\|_{L^2(\OmG)}^2 + \sum_\pm \|\sqrt{\tau}(\wideparen{v}')^\pm\|_{L^2(\Gamma)}^2}}\\
     &\leq \|f_u\|_{\cQ_\tau} + \sum\nolimits_{\pm} \sqrt{\tau}\|\wideparen{v}^\pm - \wideparen{\nu}\|_{L^2(\Gamma)}
     \ .
     \end{aligned}
     \label{eq:HDD:Galerkin:general:Uniqueness:divq}
  \end{align}
   In combination with the other four steps we find~\eqref{eq:HDD:Galerkin:inf-sup:proof:2}.
\end{proof}

\begin{theorem}[Well-posedness] \label{lem:HDD:Galerkin:Wellposedness}
  Let the assumptions of Proposition~\ref{lem:HDD:Galerkin:inf-sup} hold. Then,
  for any $f \in L^2(\OmG)$ there exists a unique solution $(\wideparen{\vq}, \wideparen{u}, \wideparen{\mu}) \in \wideparen{\cW} \times \wideparen{\cQ} \times \wideparen{\cM}$ of~\eqref{dHDD:eq:var} which satisfies with a constant $C$ independent of $\tau$ that
  \begin{subequations}
  	\label{eq:HDD:Galerkin:Wellposedness}
  \begin{align}
     \|\wideparen{\vq}\|_{\cW_\tau} + \|\wideparen{u}\|_{\cQ_{\tau/(1+\tau)}} + \|\wideparen{\mu}\|_{L^2(\Gamma)} &\leq C \|f\|_{L^2(\OmG)},\label{eq:HDD:Galerkin:NormEstqumu}
  \intertext{
  and
  }
   \big\|\jump{\wideparen{\vq}\cdot\bvec{n}}\big\|_{L^2(\Gamma)} &\leq C \sqrt{\tau} \|f\|_{L^2(\Omega)},\label{eq:HDD:Galerkin:NormEstqn}\\
	  \sqrt{\tau}
	  \left(
	  \big\|\jump{\wideparen{u}}\big\|_{L^2(\Gamma)} + \big\| \mean{\wideparen{u}} - \wideparen{\mu}\big\|_{L^2(\Gamma)}
	  \right)
	  &\leq C \|f\|_{L^2(\Omega)}.\label{eq:HDD:Galerkin:NormEstjumpu}
  \end{align}
  \end{subequations}
\end{theorem}
\begin{proof}
  The well-posedness with the stability estimate~\eqref{eq:HDD:Galerkin:NormEstqumu} follows directly from Proposition~\ref{lem:HDD:Galerkin:inf-sup}, and it remains only to show the improved estimates~\eqref{eq:HDD:Galerkin:NormEstqn} and~\eqref{eq:HDD:Galerkin:NormEstjumpu}.
  Similarly to step 2 in the proof of Proposition~\ref{lem:HDD:inf-sup} testing~\eqref{dHDD:eq:var} with $(\wideparen{\vw}, \wideparen{v}, \wideparen{\nu}) = (\wideparen{\vq}, \wideparen{u}, \wideparen{\mu})$ and using~\eqref{eq:HDD:Galerkin:NormEstqumu}
  we find
  \begin{align}
     \label{eq:HDD:Galerkin:Wellposedness:proof:1}
     \sum\nolimits_\pm \sqrt{\tau}\|\wideparen{u}^\pm - \mu\|_{L^2(\Gamma)} \leq C \|f\|_{L^2(\OmG)}.
  \end{align}
  Now, testing~\eqref{dHDD:eq:var} with $(\zerobf, 0, \jump{\wideparen{\vq}\cdot\vn})$ which is possible due to assumption $\jump{\wideparen{\cW}\cdot\vn} \subseteq \wideparen{\cM}$ and using~\eqref{eq:HDD:Galerkin:Wellposedness:proof:1} implies~\eqref{eq:HDD:Galerkin:NormEstqn}.
  Finally, we have
  \begin{align*}
      \sqrt{\tau}
	  &\left(
	  \big\|\jump{\wideparen{u}}\big\|_{L^2(\Gamma)} + \big\| \mean{\wideparen{u}} - \wideparen{\mu}\big\|_{L^2(\Gamma)}
	  \right)\\
	  &\leq \sqrt{\tau}\left(\big\| (\wideparen{u}^+ - \wideparen{\mu}) - (\wideparen{u}^- - \wideparen{\nu})\big\|_{L^2(\Gamma)}
	  + \tfrac12\big\| (\wideparen{u}^+ - \wideparen{\mu}) + (\wideparen{u}^- - \wideparen{\nu})\big\|_{L^2(\Gamma)}
	  \right)\\
	  &\leq \tfrac32\sum\nolimits_\pm \sqrt{\tau}\|\wideparen{u}^\pm - \mu\|_{L^2(\Gamma)}
  \end{align*}
  and using~\eqref{eq:HDD:Galerkin:Wellposedness:proof:1} we find~\eqref{eq:HDD:Galerkin:NormEstjumpu}.
\end{proof}

\subsection{Error analysis}
With the inf-sup conditions we obtain estimates for the perturbed Galerkin formulation that contains a consistency error contribution.

\begin{proposition}[Estimation by the best approximation error]\label{HDD:prop:Galerkin:ErrEst}
  Let the assumptions of Proposition~\ref{lem:HDD:Galerkin:inf-sup} hold, then the solutions $(\vq, u, \mu) \in \cW_\tau \times \cQ_{\tau/(1+\tau)} \times L^2(\Gamma)$ of~\eqref{HDD:eq:var} and $(\wideparen{\vq}, \wideparen{u}, \wideparen{\mu}) \in \wideparen{\cW} \times \wideparen{\cQ} \times \wideparen{\cM}$ of~\eqref{dHDD:eq:var}
  satisfy with a constant $C$ independent of $\tau$ that
  \begin{multline}
    \label{eq:HDD:Galerkin:Error}
     \|\vq - \wideparen{\vq}\|_{\cW_\tau} + \|u - \wideparen{u}\|_{\cQ_{\tau/(1+\tau)}} + \|\mu - \wideparen{\mu}\|_{L^2(\Gamma)} \\
     \leq
     C
     \inf_{(\wideparen{\vw}, \wideparen{v}, \wideparen{\nu})\in \wideparen{\cW} \times \wideparen{\cQ} \times \wideparen{\cM}}
     \big(\sqrt{1+\tau}
     \pnorm{(\vq - \wideparen{\vw},u - \wideparen{v},\mu - \wideparen{\nu})}
     + \sqrt{\tau} \inf_{\wideparen{v}' \in \wideparen{\cQ}} \sum_\pm \| \wideparen{\nu} - (\wideparen{v}')^\pm\|_{L^2(\Gamma)}
     \big)\ .
  \end{multline}
\end{proposition}
\begin{proof}
   The bilinear form $\wideparen{b}$ is continuous with a constant $C\sqrt{1+\tau}$ where the $\tau$ dependency is dominated by the term $\tau\left<\wideparen{\nu}, \wideparen{\nu}'\right>$. As the inf-sup constant of $\wideparen{b}$ is independent of $\tau$ (see Proposition~\ref{lem:HDD:Galerkin:inf-sup}) we find by Strang's lemma that
  \begin{multline}
     \label{eq:HDD:Galerkin:Error:1}
     \pnorm{(\vq - \wideparen{\vq},u - \wideparen{u},\mu - \wideparen{\mu})}
     \leq
     C
     \bigg(\sqrt{1+\tau}
     \inf_{(\wideparen{\vw}, \wideparen{v}, \wideparen{\nu})\in \wideparen{\cW} \times \wideparen{\cQ} \times \wideparen{\cM}}
     \pnorm{(\vq - \wideparen{\vw},u - \wideparen{v},\mu - \wideparen{\nu})} \\
     + \hspace{-2em}
     \sup_{(\wideparen{\vw}'', \wideparen{v}'', \wideparen{\nu}'') \in \wideparen{\cW} \times \wideparen{\cQ} \times \wideparen{\cM}}
     \hspace{-2em}
     \frac{|\blf{b}((\wideparen{\vw}, \wideparen{v}, \wideparen{\nu}), (\wideparen{\vw}'', \wideparen{v}'', \wideparen{\nu}''))
     - \wideparen{\blf{b}}((\wideparen{\vw}, \wideparen{v}, \wideparen{\nu}), (\wideparen{\vw}'', \wideparen{v}'', \wideparen{\nu}''))|}{\pnorm{(\wideparen{\vw}'', \wideparen{v}'', \wideparen{\nu}'')}}
     \bigg)\ .
  \end{multline}
  The nominator of the second term is
  \begin{align*}
     |\blf{b}((\wideparen{\vw}, &\wideparen{v}, \wideparen{\nu}), (\wideparen{\vw}'', \wideparen{v}'', \wideparen{\nu}''))
     - \wideparen{\blf{b}}((\wideparen{\vw}, \wideparen{v}, \wideparen{\nu}), (\wideparen{\vw}'', \wideparen{v}'', \wideparen{\nu}''))|\\
     &\leq \sum\nolimits_\pm \tau \left| \big<\wideparen{\nu}, (\Div \wideparen{\vw}'')^\pm \big>
           - \big< (\Div \wideparen{\vw})^\pm, \wideparen{\nu}''\big>\right| \\
     &\leq \sum\nolimits_\pm \tau \left( \left| \big<\wideparen{\nu} - (\wideparen{v}')^\pm, (\Div \wideparen{\vw}'')^\pm \big>\right|
          + \left|\big< (\Div \wideparen{\vw})^\pm - (\Div \vq)^\pm, \wideparen{\nu}''\big>\right|\right)
  \end{align*}
  for any $\wideparen{v}' \in \wideparen{\cQ}$, where we used~\eqref{eq:wideparencQ:dual:1} and $(\Div \vq)^\pm = 0$ by Theorem~\ref{lem:HDD:Wellposedness}. %
  The second term can be incorporated in the first term of~\eqref{eq:HDD:Galerkin:Error:1}
  and using again~\eqref{eq:wideparencQ:dual:1}  it can be seen even as an error contribution of higher order. Using the Cauchy-Schwarz inequality and the definition of the $\pnorm{\cdot}$ norm we obtain~\eqref{eq:HDD:Galerkin:Error}, which completes the proof.
\end{proof}
In Proposition~\ref{HDD:prop:Galerkin:ErrEst} the discretization error is bounded by the best approximation error which we would like to bound by the interpolation error for the considered finite element methods in Sec.~\ref{sec:FEM}.
However, the error of solution component $(\Div \wideparen{\vw})^\pm$ on~$\Gamma$ in the $\pnorm{\cdot}$-norm can not directly
be estimated as it is implicitly determined with the conditions~\eqref{eq:wideparencQ:dual:divq_paren}. In the following lemma we give an estimate without divergence terms on $\Gamma$ that can be used with interpolation operators. %
For this we consider more generally the affine space $\wideparen{\cW}_f$ where in difference to $\wideparen{\cW}$ for each $\wideparen{\vw} \in \wideparen{\cW}_f$ the components $(\Div \wideparen{\vw})^\pm$ fulfill
\begin{align}
  \sum_\pm \tau \big<(\Div \wideparen{\vw})^\pm, v^\pm\big> &= \big< f - (\Div \wideparen{\vw})^\circ, v^\circ\big>
  \qquad \forall v \in \wideparen{Q}^\perp\
  \label{eq:wideparencQ:dual:2:f}
\end{align}
instead of~\eqref{eq:wideparencQ:dual:2}. %
For a solution $(\wideparen{\vq}, \wideparen{u}, \wideparen{\mu}) \in \wideparen{\cW} \times \wideparen{\cQ} \times \wideparen{\cM}$ of~\eqref{dHDD:eq:var} we may consider $(\wideparen{\vq}_f, \wideparen{u}, \wideparen{\mu}) \in \wideparen{\cW} \times \wideparen{\cQ} \times \wideparen{\cM}$ where $\wideparen{\vq}_f$ differs from $\wideparen{\vq}$ only in the divergence components on $\Gamma$ that fulfill~\eqref{eq:wideparencQ:dual:2:f} instead of~\eqref{eq:wideparencQ:dual:2}.

\begin{lemma}[Estimation of the best approximation error]\label{HDD:them:Galerkin:ErrEst}
  Let the assumptions of Proposition~\ref{lem:HDD:Galerkin:inf-sup} hold and $f \in Q_\tau$, then for the solutions $(\vq, u, \mu) \in \cW_\tau \times \cQ_{\tau/(1+\tau)} \times L^2(\Gamma)$ of~\eqref{HDD:eq:var} and any $(\wideparen{\vw}_f, \wideparen{v}, \wideparen{\nu}) \in \wideparen{\cW}_f \times \wideparen{\cQ} \times \wideparen{\cM}$ it holds with a constant $C$ independent of $\tau$ that
  \begin{multline}
    \label{eq:HDD:Galerkin:Error:Improve}
     \pnorm{(\vq - \wideparen{\vw}_f,u - \wideparen{v},\mu - \wideparen{\nu})}^2
     =
     \big\|\vq - \wideparen{\vw}_f \big\|_{H(\Div^\circ,\OmG)}^2 + \tfrac{1}{\sqrt{1+\tau}} \big\| \jump{\wideparen{\vw}_f\cdot\vn} \big\|_{L^2(\Gamma)}^2 \\
     + \big\|u - \wideparen{v}\big\|_{\cQ_{\tau/(1+\tau)}}^2
     \hspace{-0.5em}
     + \big\| \mu - \wideparen{\nu} \big\|_{L^2(\Gamma)}^2.
  \end{multline}
\end{lemma}
\begin{proof}
   First we find for the $\cQ_\tau$-norm using $(\Div \vq)^\circ = -f$ and $\sqrt{\tau}(\Div \vq)^\pm = 0$ that
   \begin{align*}
      \| \Div(\vq &- \wideparen{\vw}_f) \|_{\cQ_\tau}
      = \sup_{v \in \cQ_\tau \setminus \{ 0 \} } \frac{ \left<-f - (\Div \wideparen{\vw}_f)^\circ), v^\circ\right> - \sum_\pm \tau \left<(\Div \wideparen{\vw}_f)^\pm, v^\pm \right>}{\|v\|_{\cQ_\tau}} \\
      &= \sup_{(\wideparen{v},\wideparen{v}^\bot) \in \wideparen{\cQ} \times \wideparen{\cQ}^\bot \setminus \{ (0,0) \}}
      \frac{ \left<-f - (\Div \wideparen{\vw}_f)^\circ, \wideparen{v}^\circ\right> - \sum_\pm \tau \left<(\Div \wideparen{\vw}_f)^\pm, \wideparen{v}^\pm \right>
      - \left< f, (\wideparen{v}^\bot)^\circ\right>
      - \left< \Div \wideparen{\vw}_f, \wideparen{v}^\bot\right>_{\cQ_\tau}}{\|\wideparen{v} + \wideparen{v}^\bot\|_{\cQ_\tau}} \\
      &= \sup_{(\wideparen{v},\wideparen{v}^\bot) \in \wideparen{\cQ} \times \wideparen{\cQ}^\bot \setminus \{ (0,0) \} }
      \frac{ \left<-f - (\Div \wideparen{\vw}_f)^\circ, \wideparen{v}^\circ\right> }{\|\wideparen{v} + \wideparen{v}^\bot\|_{\cQ_\tau}} \\
      &= \| -f - \Div \wideparen{\vw}_f \|_{L^2(\OmG)} = \| \Div (\vq - \wideparen{\vw}_f) \|_{L^2(\OmG)},
   \end{align*}
   where we used that $\sum_\pm \tau \left<(\Div \wideparen{\vw}_f)^\pm, \wideparen{v}^\pm \right> = 0$ by~\eqref{eq:wideparencQ:dual:1} and
   $\left<\Div \wideparen{\vw}_f, \wideparen{v}^\bot\right>_{\cQ_\tau} = \left< f, (\wideparen{v}^\bot)^\circ\right>$ by \eqref{eq:wideparencQ:dual:2:f}. %
   We find~\eqref{eq:HDD:Galerkin:Error:Improve} by adding the other terms in the $\pnorm{\cdot}$-norm.
\end{proof}

\section{Finite element discretization}
\label{sec:FEM}
\subsection{Introduction of the formulation with traces and projected traces}
For any $v$ in the Sobolev space $\cQ_\tau$ in the variational formulation~\eqref{HDD:eq:var} the components $\sqrt{\tau} v^\pm \in L^2(\Gamma)$ are independent of $v^\circ \in L^2(\OmG)$, where for a more regular solution $u \in H^{\nicefrac32}(\OmG)$ the components $\sqrt{\tau} u^\pm$ are the weighted traces $\sqrt{\tau} \tr^\pm u$.
When constructing a finite element method, we choose the discrete space $M^q(\mathcal{F}_{\Gamma,h})$ for the hybrid variable $\mu_h$ on $\Gamma$ as well as $Q^q(\mathcal{T}_h)$ and $W^q(\mathcal{T}_h)$ for the solution components $u_h$ and $\vq_h$ in $\OmG$, respectively. Note that the latter two fulfill a discrete de~Rham sequence. Then, in difference to the variational formulation we may fix for any $v_h \in Q^q(\mathcal{T}_h)$ -- in the context of finite element methods we prefer to omit the superscript in $v_h^\circ$ -- the associated components $\sqrt{\tau}(v_h)^\pm$ on $\Gamma$. %
In this paper, we choose $\sqrt{\tau} \pitr^\pm v_h$ with projected one-sided traces $\pitr^\pm$.
In contrast, following the Galerkin formulations from Sec.~\ref{sec:Galerkin}, the contributions $(\Div \vw_h)^\pm$ on $\Gamma$ for any $\vw_h \in W^q(\mathcal{T}_h)$ are determined by~\eqref{eq:wideparencQ:dual:divq_paren}. %

\subsubsection{Conforming finite element formulation with projected traces}
We have chosen piecewise polynomials $M^q(\mathcal{F}_{\Gamma,h})$ but the traces $\tr^\pm v_h$ on $\Gamma$ of functions $v_h \in Q^q(\mathcal{T}_h)$ are not necessarily polynomials on the facets $F\in\mathcal{F}_h$. %
More precisely, on a cell $K \in \mathcal{T}_h$ we have
\begin{align}
  v_h|_K(\bvec{x}) = |\det \nabla \Phi_K (\widehat{\bvec{x}})|^{-1} \widehat{v}_h(\widehat{\bvec{x}}) \text{ for } \bvec{x} = \Phi_K (\widehat{\bvec{x}})
  \label{FEM:eq:vh_transformation}
\end{align}
and some $\widehat{v}_h \in \mathcal{P}^q$ such that
\begin{align*}
   \int_K v_h \text{d}\bvec{x}
   = \int_{\widehat{K}} \widehat{v}_h \,\text{d}\widehat{\bvec{x}}.
\end{align*}
For each face $F \in \mathcal{F}_{\Gamma,h}$ we have two neighboring cells $K_F^\pm \in \mathcal{T}_h$ and the trace of $v_h|_{K_F^\pm}$ to $F$ is given by
\begin{align*}
   \tr^\pm v_h(\bvec{x}) = |\det \nabla \Phi_{K_F^\pm} (\Phi_{K_F^\pm}^{-1}\Phi_F \widehat{\bvec{x}})|^{-1} \tr^\pm\widehat{v}_h (\widehat{x}_1,0) \text{ for } \bvec{x} = \Phi_F \widehat{x}_1 \in F
\end{align*}
with the mapping $\Phi_F: [0,1] \to F$, where we assume without loss of generality that $\Phi_F \widehat{x}_1 = \Phi_{K^\pm_F}(\widehat{x}_1, 0)$, i.e., it is the lower edge in the reference quad $[0,1]^2$.
With this we can define the projected traces
\begin{align}
    (\pitr^\pm v_h\big|_F\circ \Phi_F, \widehat{\nu}_h)_{L^2(0,1)} = (\tr^\pm v_h\big|_F\circ \Phi_F, \widehat{\nu}_h)_{L^2(0,1)}
    \quad \forall \widehat{\nu}_h \in \mathcal{P}^q.´
   \label{FEM:eq:Pi_h}
\end{align}
If we then expand $\pitr^\pm_h v|_F \circ \Phi_F$ in Legendre polynomials then the left-hand side becomes a diagonal matrix that is identical for all $F \in \mathcal{F}_{\Gamma,h}$, since $M^q(\mathcal{F}_{\Gamma,h})$ restricted to a facet $F$ and evaluated on the reference interval coincides with $\mathcal{P}^q$.
\begin{lemma}[Projection error]
If the mesh is quasi-uniform and shape-regular then there holds for all $v_h \in Q^q(\mathcal{T}_h)$ with a constant $C > 0$
\begin{align}
  \sum\nolimits_\pm \sqrt{\tau}\|\pitr^\pm v_h - \tr^\pm v_h\|_{L^2(\Gamma)}
  \leq C \sqrt{\tau} h^{q+\nicefrac12} \sum\nolimits_\pm \big(\sum_{F \in \mathcal{F}_{\Gamma,h}}| \tr^\pm v_h|_{H^{q+1}(F)}^2\big)^{\nicefrac12}.
\end{align}
\label{lem:FEM:projection_error}
\end{lemma}
\begin{proof}
   The error of the weighted $L^2(\Gamma)$-projection~\eqref{FEM:eq:Pi_h} can be bounded as
   \begin{align*}
      \|\pitr^\pm v_h &- \tr^\pm v_h\|_{L^2(\Gamma)}^2
      = \sum_{F \in \mathcal{F}_{\Gamma,h}}
      \|\pitr^\pm v_h - \tr^\pm v_h\|_{L^2(F)}^2 \\
      &= \sum_{F \in \mathcal{F}_{\Gamma,h}} \| (\pitr^\pm v_h - \tr^\pm v_h) \circ \Phi_F |\Phi_F'|^{\nicefrac12} \|_{L^2(0,1)}^2\\
      &\leq C h \sum_{F \in \mathcal{F}_{\Gamma,h}} \| (\pitr^\pm v_h - \tr^\pm v_h) \circ \Phi_F \|_{L^2(0,1)}^2 \\
      &= C h \sum_{F \in \mathcal{F}_{\Gamma,h}} \|(\widehat{\Pi}- Id)
      |\det \nabla \Phi_{K_F^\pm}|^{-1}\big|_F
      \tr^\pm\widehat{v}_h(\cdot,0)\|_{L^2(0,1)}^2\\
      &\leq C h^{-1}\sum_{F \in \mathcal{F}_{\Gamma,h}} \|(\widehat{\Pi}- Id)
      \tr^\pm\widehat{v}_h(\cdot,0)\|_{L^2(0,1)}^2 \\
      &\leq C h^{-1}\sum_{F \in \mathcal{F}_{\Gamma,h}} | \tr^\pm\widehat{v}_h(\cdot,0)|_{H^{q+1}(0,1)}^2
      \leq C h^{2q+1}\sum_{F \in \mathcal{F}_{\Gamma,h}} |\tr^\pm v_h|_{H^{q+1}(F)}^2,
   \end{align*}
   where $\widehat{v}_h$ was defined in~\eqref{FEM:eq:vh_transformation}.
\end{proof}

\subsection{Well-posedness}
In this section we show that the finite element formulation with projections~\eqref{HDD:eq:dHDD:FEM:proj} fulfills the assumptions of Theorem~\ref{lem:HDD:FEM:Wellposedness} and belongs to the considered well-posed perturbed Galerkin methods.
\begin{proof}[Proof of Theorem~\ref{lem:HDD:FEM:Wellposedness}]
  The mapping $\Div: W^q(\cT_h) \to Q^q(\cT_h)$ is surjective due to the commuting diagram property~\cite{Bespalov.Heuer:2011} and the surjectivity of $\Div: H(\Div,\OmG) \to L^2(\OmG)$, and so $\Div W^q(\cT_h) = Q^q(\cT_h)$. Moreover, as $\cT_h$ is a conforming mesh of $\Omega$ it holds $\jump{W^q(\cT_h)\cdot\vn} = M^q(\mathcal{F}_{\Gamma,h})$. %
  We set
  \begin{align*}
    \wideparen{\cQ} = \cQ^q(\cT_h) + \sum\nolimits_\pm \sqrt{\tau} \pitr^\pm \cQ^q(\cT_h) \circ  (\jmath^\pm)^{-1}, \quad
    \wideparen{\cM} = M^q(\mathcal{F}_{\Gamma,h}),
  \end{align*}
  i.e., it is $\wideparen{\cQ}^\circ = \cQ^q(\cT_h)$ and $\sqrt{\tau}\wideparen{\cQ}^\pm = \sqrt{\tau}\pitr^\pm\cQ^q(\cT_h)$.
  Then, with the Banach space adjoint $\wideparen{\cQ}'$ defined by~\eqref{eq:wideparencQ:dual}, we set
  \begin{align*}
    \wideparen{\cW} = \{ \vw \in W^q(\cT_h): \Div \vw \in \wideparen{\cQ}' \}
  \end{align*}
  i.e., $(\Div \wideparen{\cW})^\circ = \Div W^q(\cT_h)$ and $(\Div \wideparen{\cW})^\pm$ is determined from $(\Div \wideparen{\cW})^\circ$ by~\eqref{eq:wideparencQ:dual:2} and~\eqref{eq:wideparencQ:dual:1}.
  The projection based interpolation operator \[\Pi^{\Div}_{q+1}: H^r(\OmG)^2 \times H(\Div,\OmG) \to W^q(\cT),\ r > 0,\]  defined in~\cite{Bespalov.Heuer:2011} satisfies the commuting diagram property (de~Rham diagram) and is stable for any $r > 0$. %
  With these interpolation operators, the space $\wideparen{\cW}$ fulfills the assumptions of
  Lemma~\ref{lem:HDD:Galerkin:deRham} and Proposition~\ref{lem:HDD:Galerkin:inf-sup}.
  Note, that similar properties have been shown for the interpolation operator to edge element spaces in~\cite{Demkowicz.Babuska:2003}, which are relevant here due to the isomorphism of $\Div$ and $\operatorname{rot}_{2D}$ in $\IR^2$.
  Hence,~Theorem~\ref{lem:HDD:Galerkin:Wellposedness} applies and a unique solution of~\eqref{HDD:eq:dHDD:FEM:proj} exists. Finally,~\eqref{eq:HDD:Galerkin:Wellposedness} implies~\eqref{eq:HDD:FEM:Wellposedness}, and the proof is complete.
\end{proof}

\subsection{Error analysis}
Let $I_h$ denote the $L^2(\OmG)$ projection onto $Q^q(\mathcal{T})$.
\begin{lemma}[Best-approximation error in $Q_{\tau/(1+\tau)}$]
   Let $\mathcal{T}_h$ be quasi-uniform and shape-regular, $q\in\mathbb{N}_0$ and $u \in H^s(\OmG)$, $s > \tfrac12$.
   Then, there exists constants $C_1, C_2$ such that
   \begin{align}
      \inf_{v_h \in Q^q(\mathcal{T}_h)} \| u - v_h\|_{L^2(\OmG)} + \sum_{\pm}\tfrac{\sqrt{\tau}}{\sqrt{1+\tau}} \|\tr^\pm u - \pitr^\pm v_h\|_{L^2(\Gamma)}
      \leq \operatorname{err}(h,q,\tau,u,C_1,C_2)\ .
      \label{eq:FEM:bestapproximation_error:Q_t}
   \end{align}
   \label{lem:FEM:bestapproximation_error:Q_t}
\end{lemma}
\begin{proof}
   Using the trace theorem we find for any $v_h \in Q^q(\mathcal{T})$ that
   \begin{align*}
      \sum_{\pm}\sqrt{\tau} \|\tr^\pm (u - v_h)\|_{L^2(\Gamma)}
      \leq C\sqrt{\tau} \|u - v_h\|_{H^{\nicefrac12}(\OmG)}.
   \end{align*}
   Choosing $v_h = I_h u$ we find that
   \begin{align*}
      \inf_{v_h \in Q^q(\mathcal{T}_h)}
      \|u - v_h\|_{L^2(\OmG)} +
      \tfrac{\sqrt{\tau}}{\sqrt{1+\tau}} \|u - v_h\|_{H^{\nicefrac12}(\OmG)}
      \leq \operatorname{err}(h,q,\tau,u,C_1,C_2)
   \end{align*}
   for some $C_1, C_2, > 0$. Finally, using Lemma~\ref{lem:FEM:projection_error}, the trace theorem and the stability of $I_h$ in $H^{s}(\OmG)$ we conclude in~\eqref{eq:FEM:bestapproximation_error:Q_t}.
\end{proof}

\begin{lemma}[Consistency error]
   Let $\mathcal{T}_h$ be quasi-uniform and shape-regular, $q\in\mathbb{N}_0$ and $u \in H^s(\OmG)$, $s > \tfrac12$.
   Then, there exists a constant $C_2$ such that
  \begin{align}
     \inf_{v_h'\in Q^q(\cT_h)} \sum_\pm \sqrt{\tau} \|\pitr^\pm (u - v_h')\|_{L^2(\Gamma)} \leq \sqrt{1+\tau} \operatorname{err}(h,q,\tau,u,0,C_2)\ .
     \label{eq:FEM:consistencyerror}
  \end{align}
  \label{lem:FEM:consistencyerror}
\end{lemma}
\begin{proof}
   With the triangle inequality and the trace theorem we have
   \begin{align*}
      \inf_{v_h'\in Q^q(\cT_h)} \sum_\pm \sqrt{\tau} \|\pitr^\pm (u - v_h')\|_{L^2(\Gamma)}
      &\leq C \inf_{v_h'\in Q^q(\cT_h)} \sum_\pm \sqrt{\tau} \|\tr^\pm (u - v_h')\|_{L^2(\Gamma)} \\
      &\leq C \inf_{v_h'\in Q^q(\cT_h)} \sum_\pm \sqrt{\tau} \| u - v_h'\|_{H^{\nicefrac12}(\OmG)}.
   \end{align*}
   Finally, inserting $v_h' = I_h u$ we conclude in~\eqref{eq:FEM:consistencyerror}.
\end{proof}

\begin{proof}[Proof of Theorem~\ref{lem:FEM:discretisation_error}]
   For the solution $(\vq_h, u_h, \mu_h) \in W^q(\mathcal{T}_h) \times Q^q(\mathcal{T}_h) \times M^q(\mathcal{F}_{\Gamma,h})$, $q\in\mathbb{N}_0$, of the conforming formulation~\eqref{HDD:eq:dHDD:FEM:proj} it holds~\eqref{eq:HDD:Galerkin:Error} for
   \begin{align*}
      \wideparen{u} &= u_h + \sum\nolimits_\pm \sqrt{\tau} \pitr^\pm u_h \circ (\jmath^\pm)^{-1} \in \wideparen{\cQ}, \\
      \wideparen{\vq} &= \vq_h \text{ in } \OmG \text{ with }
      \Div\wideparen{\vq} \in \wideparen{\cW}_f, \\
      \wideparen{\mu} &= \mu_h \in \wideparen{\cM},
   \end{align*}
   where $\wideparen{\cQ}$, $\wideparen{\cW}$ and $\wideparen{\cM}$ are defined as in the proof of Theorem~\ref{lem:HDD:FEM:Wellposedness} and $\wideparen{\cW}_f$ is defined as $\wideparen{\cW}$ satisfying~\eqref{eq:wideparencQ:dual:2:f} instead of~\eqref{eq:wideparencQ:dual:2}.
   Then, using the interpolation error estimates for $\cW^q(\cT_h)$ and $\cM^q(\cF_{\Gamma,h})$, the estimates of the best approximation error in $\cQ_{\tau/(1+\tau)}$ in Lemma~\ref{lem:FEM:bestapproximation_error:Q_t} and the estimate of the consistency error in Lemma~\ref{lem:FEM:consistencyerror} we find~\eqref{eq:FEM:discretisation_error:uqmu}.
   Moreover, the remaining error terms in~\eqref{eq:HDD:Galerkin:Error} imply
   \begin{align}
      \sum_\pm \tfrac{\sqrt{\tau}}{\sqrt{1+\tau}} \| \tr^\pm u - \pitr^\pm u_h \|_{L^2(\Gamma)}
      + \tfrac{1}{\sqrt{1+\tau}} \big\|\jump{\vq_h\cdot\bvec{n}}\big\|_{L^2(\Gamma)}
      \leq \sqrt{1+\tau}\operatorname{err}(h,q,\tau, u, C_1, C_2).
      \label{eq:FEM:discretisation_error:utrace}
   \end{align}
   Now, using $\tr^+ u = \tr^-u$ on $\Gamma$, the triangle inequality and~\eqref{eq:FEM:discretisation_error:utrace} we obtain
		\begin{align*}
			\sqrt{\tau} \big\| \jump{\pitr u_h} \big\|_{L^2(\Gamma)}
			&= \sqrt{\tau} \big\| \pitr^+u_h - \tr^+u - (\pitr^-u_h - \tr^-u)  \big\|_{L^2(\Gamma)} \\
			&\leq \sum_\pm \sqrt{\tau} \| \tr^\pm u - \pitr^\pm u_h \|_{L^2(\Gamma)}
			\leq \sqrt{1+\tau}\sum_\pm \tfrac{\sqrt{\tau}}{\sqrt{1+\tau}} \| \tr^\pm u - \pitr^\pm u_h \|_{L^2(\Gamma)} \\
			&\leq (1+\tau)\operatorname{err}(h,q,\tau, u, C_1, C_2).
		\end{align*}
		Similarly, using $\tr^\pm u = \mu$ on $\Gamma$, the triangle inequality and~\eqref{eq:FEM:discretisation_error:utrace} we get
		\begin{align*}
			\sqrt{\tau} \big\| \mean{\pitr u_h} - \mu_h \big\|_{L^2(\Gamma)}
			&\leq \sqrt{1+\tau}\sum_\pm \tfrac{\sqrt{\tau}}{\sqrt{1+\tau}} \| \tr^\pm u - \pitr^\pm u_h \|_{L^2(\Gamma)}
			+ \sqrt{\tau} \|\mu - \mu_h \|_{L^2(\Gamma)} \\
			&\leq C\sqrt{1+\tau} \left( \sqrt{1+\tau} + \sqrt{\tau}\right)\operatorname{err}(h,q,\tau, u, C_1, C_2)
		\end{align*}
		and so~\eqref{eq:FEM:discretisation_error:tr_u}. Finally, testing~\eqref{HDD:eq:dHDD:FEM:proj:3} with $\nu_h = \jump{\vq_h\cdot\vn}$ we find that
		\begin{align*}
		 \big\|\jump{\vq_h\cdot\bvec{n}}\big\|_{L^2(\Gamma)} &\leq \tau \sum_\pm \big\| \pitr^\pm u_h - \mu_h \big\|_{L^2(\Gamma)} \\
		 &\leq \sqrt{\tau} \left( \sqrt{\tau}\big\| \jump{\pitr u_h} \big\|_{L^2(\Gamma)} + \sqrt{\tau} \big\| \mean{\pitr u_h} - \mu_h \big\|_{L^2(\Gamma)}\right) \\
		 &\leq \sqrt{\tau}(1+\tau)\operatorname{err}(h,q,\tau, u, C_1, C_2).
		\end{align*}
		Taking the best of this estimate and the estimate~\eqref{eq:FEM:discretisation_error:utrace} for $\|\jumpt{\vq_h\cdot\bvec{n}}\|_{L^2(\Gamma)}$ and using $\min(\sqrt{1+\tau}, \sqrt{\tau}(1+\tau)) \leq 2\sqrt{\tau}$ for all $\tau \geq 0$, we conclude in the improved estimate~\eqref{eq:FEM:discretisation_error:qn}.
\end{proof}

\section{Numerical experiments}
\label{sec:NumExs}
\begin{figure}[bt]
	\centering
	\begin{subfigure}{0.49\textwidth}
		\centering
		\includegraphics[height=4cm]{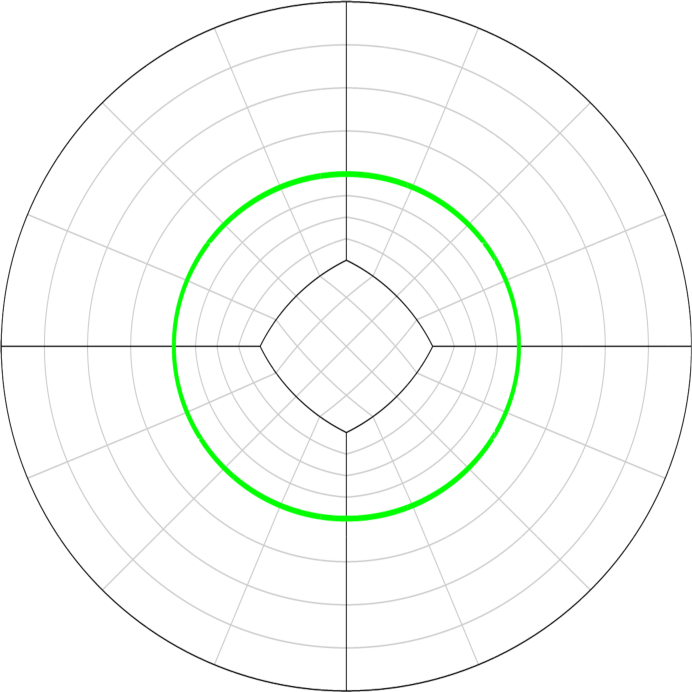}
		\label{NE:fig:cQ3_kmsh}
	\end{subfigure}
	\hfill
	\begin{subfigure}{0.49\textwidth}
		\centering
		\includegraphics[height=4cm]{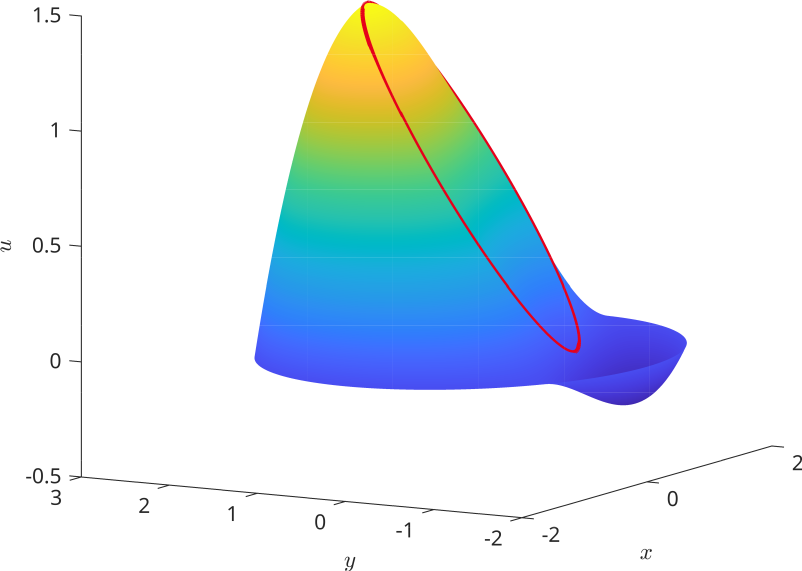}
		\label{ne:fig:cQ3_k_u}
	\end{subfigure}
	\caption{Computational domain $\Omega$ and reference solution component $u$.}
	\label{NE:fig:CMesh_Sol_u}
\end{figure}

In this last section, we investigate the finite element discretization of the HMDD formulation numerically. %
For this, we implemented the finite element method described in Sec.~\ref{sec:FEM} using the numerical \texttt{C++}-library \texttt{Concepts~2}\cite{Concepts:2023}.In our computations, we employ meshes with curved quadrilaterals and use a direct solver, either \texttt{SuperLU} or \texttt{MUMPS}, while \texttt{Concepts~2}'s \texttt{MATLAB} interface is accessed for graphical analysis and visualization of the numerical results.

\begin{figure}[tb]
	\begin{subfigure}{.32\textwidth}
		\centering
		\includegraphics[width=\textwidth]{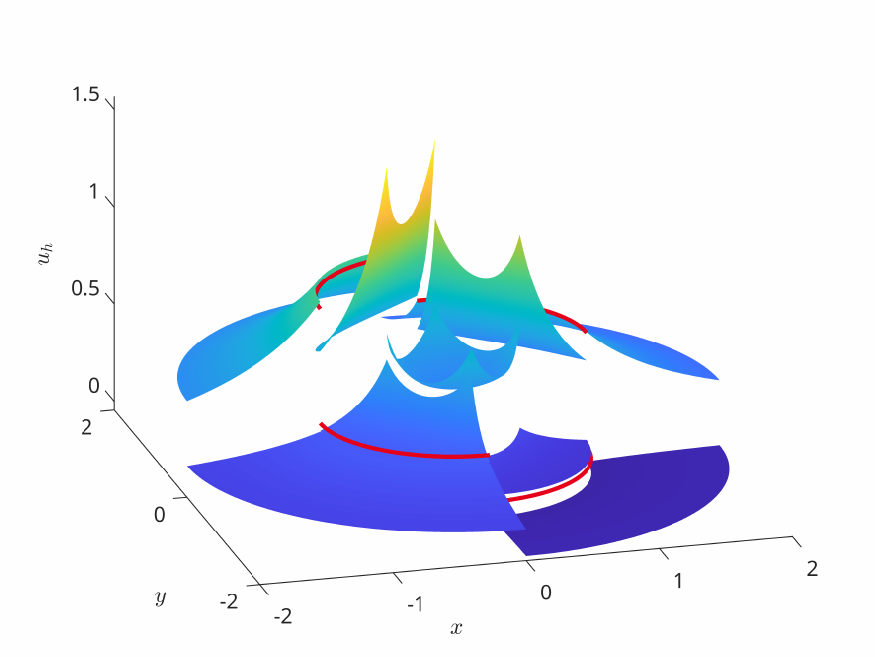}
	\end{subfigure}
	\hfill
	\begin{subfigure}{.32\textwidth}
		\centering
		\includegraphics[width=\textwidth]{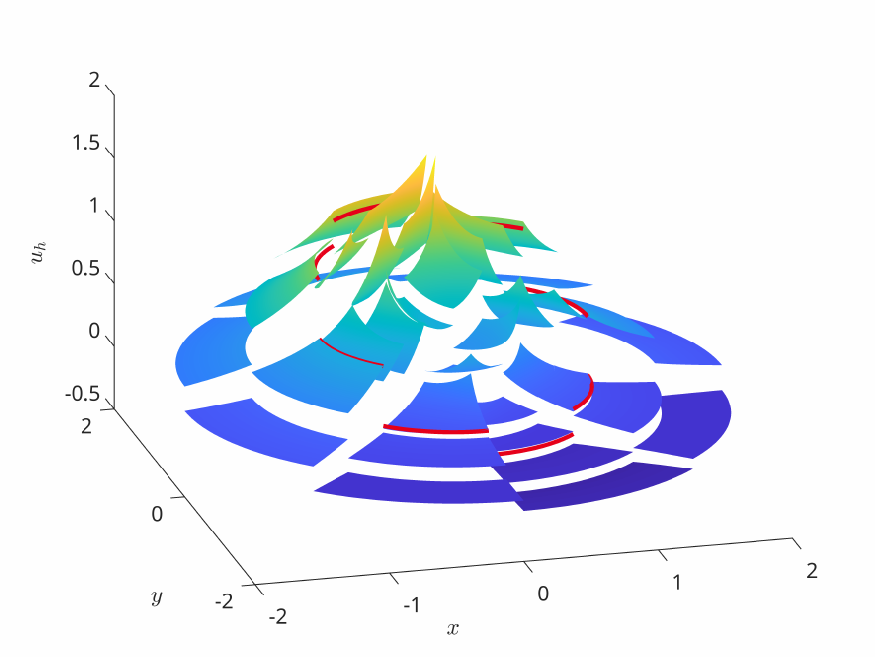}
	\end{subfigure}
	\hfill
	\begin{subfigure}{.32\textwidth}
		\centering
		\includegraphics[width=\textwidth]{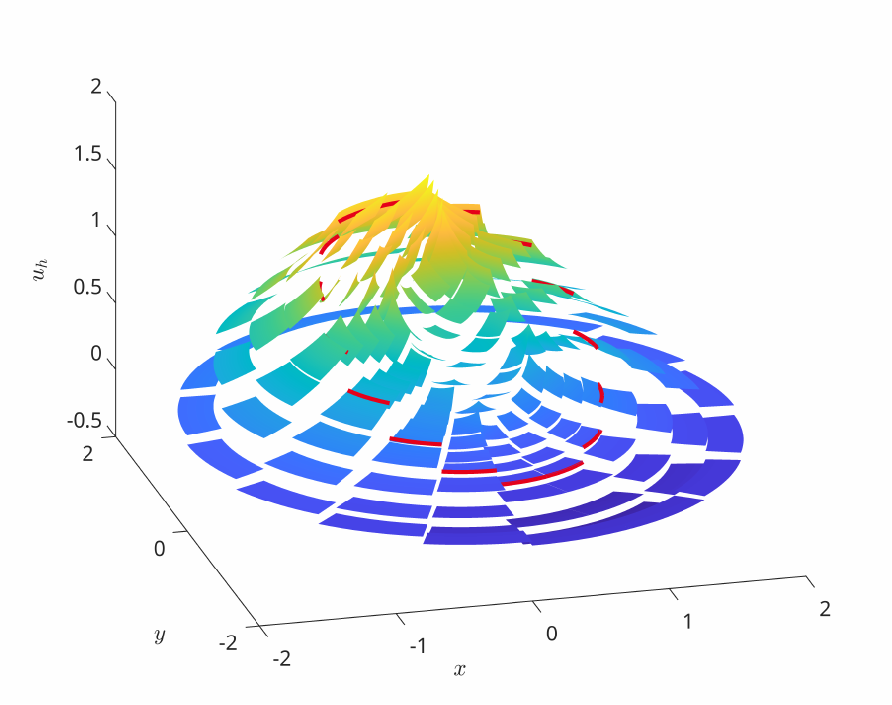}
	\end{subfigure}
	\begin{subfigure}{.32\textwidth}
		\centering
		\includegraphics[width=\textwidth]{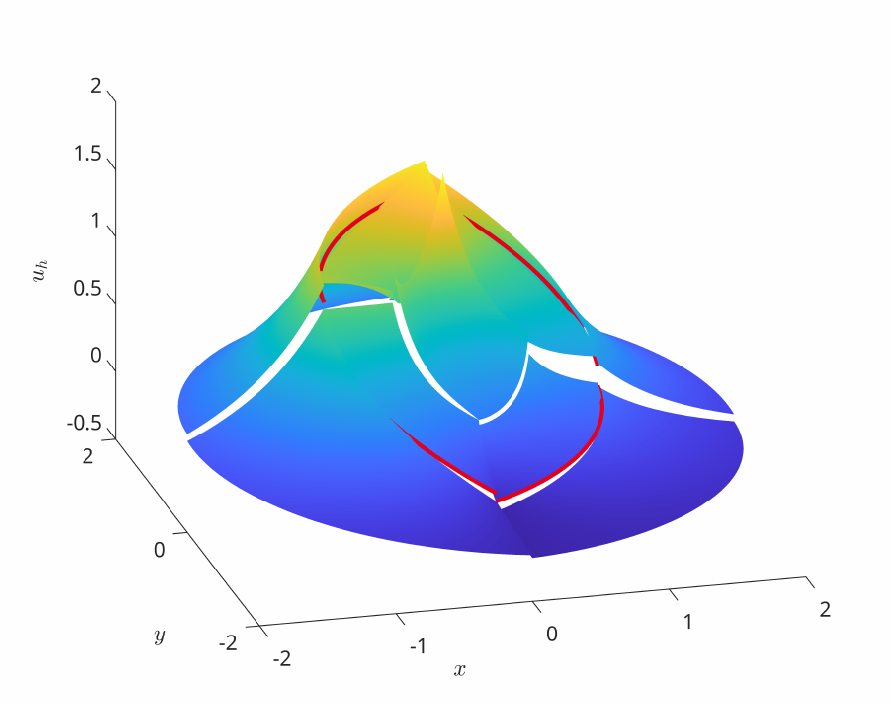}
	\end{subfigure}
	\hfill
	\begin{subfigure}{.32\textwidth}
		\centering
		\includegraphics[width=\textwidth]{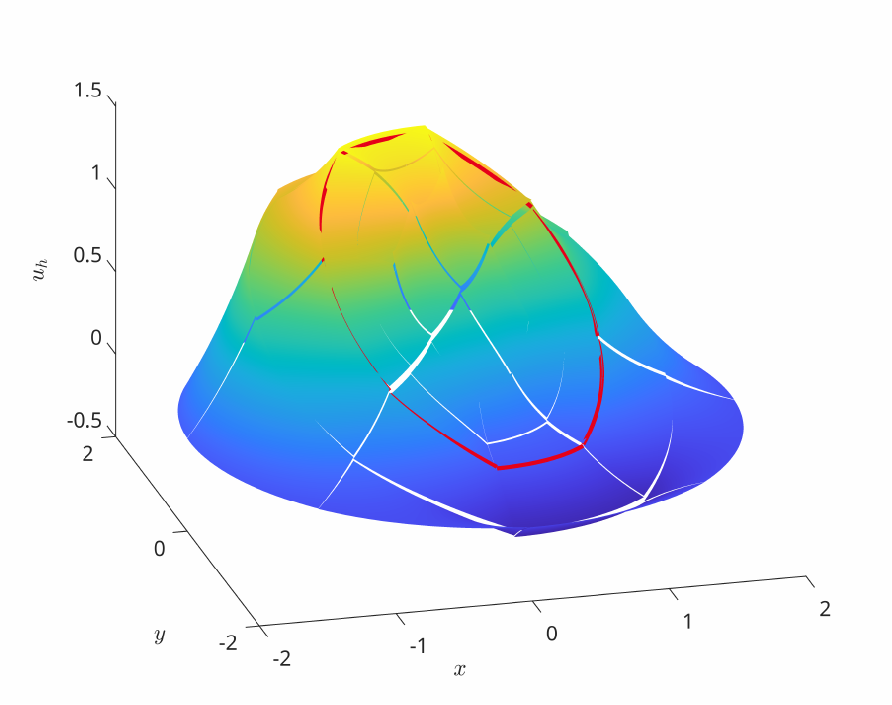}
	\end{subfigure}
	\hfill
	\begin{subfigure}{.32\textwidth}
		\centering
		\includegraphics[width=\textwidth]{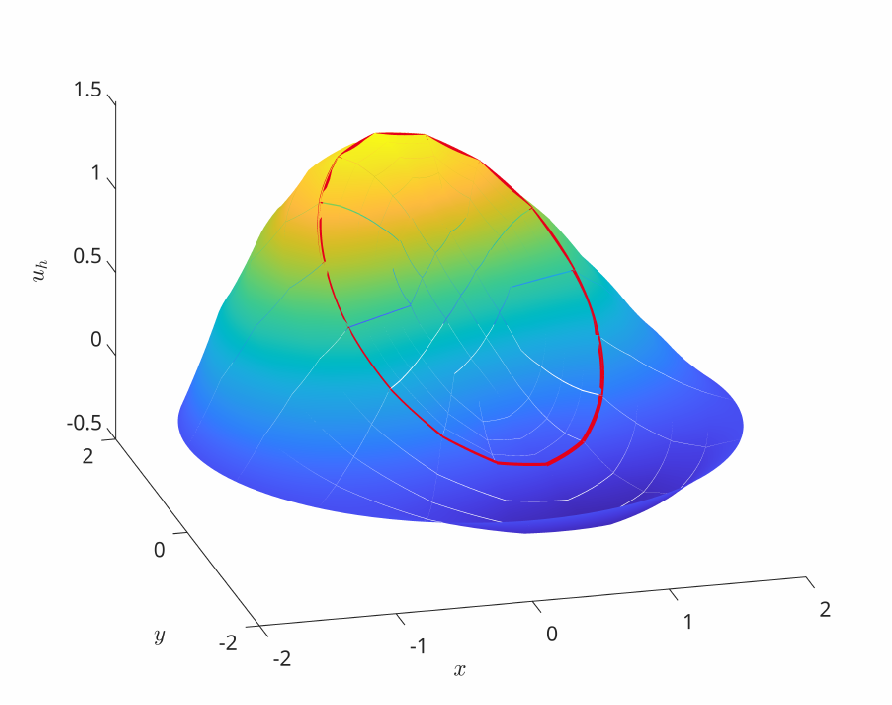}
	\end{subfigure}
	\caption{Numerical solution $u_h$ and $\mu_h$ (red lines) at polynomial order $q=0$ (upper half) and $q=1$ (lower half) for $0,\, 1$ and $2$ mesh refinements (from left to right).}
	\label{NE:fig:App_Sol_u_cQ}
\end{figure}

For our investigation, we consider a radially symmetric domain $\Omega$ of radius $2$, where the interface $\Gamma$ -- on which the hybrid variable $\mu$ is defined -- is given by the unit circle. A sketch of the geometry is shown on the left-hand side of Fig.~\ref{NE:fig:CMesh_Sol_u} with the green circle marking the interface $\Gamma$. The initial mesh is outlined in dark gray while the light gray lines show the mesh after two refinements.
Further, we consider a material function $\kappa(\vx)~=~16\chi_{[0,1)}(r)~+~1\chi_{(1,2]}(r)$ in polar coordinates with a jump at the interface and the right-hand side
\[f(\vx) =
\begin{cases}
	\frac{47}{2}\sqrt{2+\sqrt{2}}r\sin(\varphi)-\frac{47}{2}\sqrt{2-\sqrt{2}}r\cos(\varphi)+1, & 0 \leq r < 1,\\
	\sqrt{2+\sqrt{2}}r\sin(\varphi)-\sqrt{2-\sqrt{2}}r\cos(\varphi)+1, & 1 < r \leq 2.
\end{cases}\]

The corresponding reference solution component $u$, which is analytic on $\OmG$, is shown on the right-hand side of Fig.~\ref{NE:fig:CMesh_Sol_u} with the red circle being the solution component $\mu$.\\
First, we investigate the the numerical solution obtained by the HMDD method for polynomial orders $q=0,\,1$ at a fixed parameter $\tau=10$. The resulting numerical solution components $u_h$ and $\mu_h$ at mesh widths $h=2^0,\, 2^{-1},\, 2^{-2}$ are shown in Fig.~\ref{NE:fig:App_Sol_u_cQ}.
The depictions of the numerical solution components $u_h$ and $\mu_h$ (red lines) in Fig.~\ref{NE:fig:App_Sol_u_cQ} nicely illustrate the behavior of the HMDD method and indicate convergence of the numerical towards the reference solution and at a faster rate for $q = 1$.

\begin{figure}[h]
	\centering
	\begin{subfigure}{.3\textwidth}
		\centering
		\includegraphics{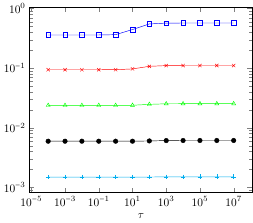}
		\caption{$\|u-u_h\|_{L^2(\Omega\setminus\Gamma)}$}
	\end{subfigure}
	\hfill
	\begin{subfigure}{.3\textwidth}
		\centering
		\includegraphics{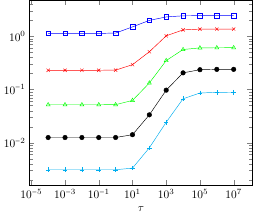}
		\caption{$\|\mathbf{q}-\mathbf{q}_h\|_{L^2(\Omega\setminus\Gamma)}$}
	\end{subfigure}
	\hfill
	\begin{subfigure}{.3\textwidth}
		\centering
		\includegraphics{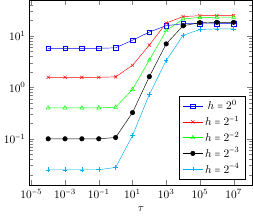}
		\caption{$\|\mathrm{div}(\mathbf{q}-\mathbf{q}_h)\|_{L^2(\Omega\setminus\Gamma)}$}
	\end{subfigure}
	\\
	\begin{subfigure}{.3\textwidth}
		\centering
		\includegraphics{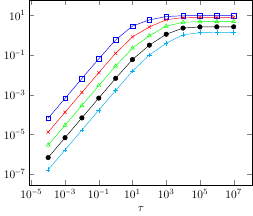}
		\caption{$\|\jump{\mathbf{q}_h\cdot\mathbf{n}}\|_{L^2(\Gamma)}$}
	\end{subfigure}
	\hfill
	\begin{subfigure}{.3\textwidth}
		\centering
		\includegraphics{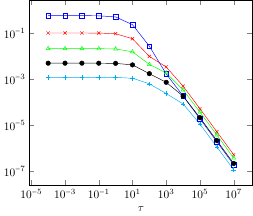}
		\caption{$\|\jump{\pitr u_h}\|_{L^2(\Gamma)}$}
	\end{subfigure}
	\hfill
	\begin{subfigure}{.3\textwidth}
		\centering
		\includegraphics{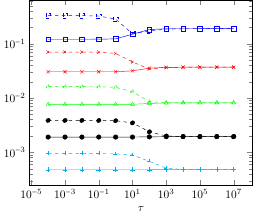}
		\caption{$\|\mu-\mu_h\|/\|\mean{u-\pitr u_h}\|$}
	\end{subfigure}
	\caption{Error plots for (a) $u_h$, (b) $\mathbf{q}_h$, (c) $\Div\mathbf{q}_h$ on $\Omega\setminus\Gamma$, (d) $\jump{\mathbf{q}_h\cdot\mathbf{n}}$, (e) $\jump{\pitr u_h}$ and (f) $\mu_h$ (solid), $\pitr u_h$ (dashed) on $\Gamma$ in dependence of $\tau$ for order $q=1$.}
	\label{NE:fig:cQ3r_l_errors_vs_tau_p1_k16}
\end{figure}

\subsection{Discretization error in dependence of $h$ and $\tau$}
We now investigate the discretization error for different refinement levels and values of $\tau$, where the results for $q=1$ are shown in Fig.~\ref{NE:fig:cQ3r_l_errors_vs_tau_p1_k16}. We have performed the experiments for different orders $q$ of the finite element method and observe the following behavior:
	\begin{subequations}
	\label{eq:NE:eoc}
	\begin{align}
		\| u - u_h \|_{L^2(\OmG)} + \| \mu - \mu_h \|_{L^2(\Gamma)} &\leq \hspace{2.9em} C_1 h^{q+1},
		\label{eq:NE:eoc:1}\\
		\| \vq - \vq_h \|_{L^2(\OmG)} &\leq \hspace{2.9em} C_1 h^{q+1} + \tfrac{C_2 h\tau}{C_3 + h\tau} h^{q+\frac12},
		\label{eq:NE:eoc:2}
		\\
		\| \Div(\vq - \vq_h) \|_{L^2(\OmG)} &\leq \hspace{2.9em}C_1 h^{q+1} + \tfrac{C_2 h\tau}{C_3 + h\tau} h^{q-\frac12},
		\label{eq:NE:eoc:3}
		\\
		\| \jump{ \vq_h\cdot \vn} \|_{L^2(\Gamma)} &\leq \hspace{7.0em}\tfrac{C_2 h \tau}{C_3 + h\tau} h^{q},
		\label{eq:NE:eoc:5}
		\\
		\| \jump{ \pitr u_h} \|_{L^2(\Gamma)} &\leq \tfrac{1}{(C_1)^{-1} + h\tau} h^{q+1},
		\label{eq:NE:eoc:6}
	\end{align}
	\label{NE:eq:conj}
\end{subequations}
	where all constants are independent of $\tau$, but may depend on $q$.\\
	First, we observe that all error contributions in~\eqref{eq:NE:eoc:2} are bounded for fixed mesh-width $h$ even for large values $\tau$.
	This means, that for large $\tau$ we do not see an increase like $\sqrt{1+\tau}$, $\tau$ or $1+\tau$ as in the error estimates~\eqref{eq:FEM:discretisation_error}.
    Then, we see a uniform convergence of $u_h$ and $\mu_h$ in $\tau$, with a rate of $q+1$ which corresponds to the estimates~\eqref{eq:FEM:discretisation_error:uqmu} for $\tau = 0$. That means that even for large values of $\tau$ the convergence rates are the optimal ones.
    For the error in $\vq_h$ and $\Div\vq_h$, we observe different convergence behavior for different values of $\tau$, which has similarities with the
    error estimate in~\eqref{eq:FEM:discretisation_error:uqmu}. Note, that $\tfrac{C_2 h\tau}{C_3 + h\tau} \to C_2$ for $\tau\to \infty$ and fixed mesh-width $h$ and $\tfrac{C_2 h\tau}{C_3 + h\tau} \to 0$ for $h \to 0$ and fixed $\tau$. %
    For the $L^2(\OmG)$-error of $\vq_h$ we observe a rate of $q+1$ for small values of $\tau$ and a reduced rate of $q+\nicefrac12$ for large values of $\tau$, which is in agreement with~\eqref{eq:FEM:discretisation_error:uqmu} without the factor $\sqrt{1+\tau}$. For the $L^2(\OmG)$-error of $\Div\vq_h$ we observe a rate of $q+1$ for small values of $\tau$ and a reduced rate of $q-\nicefrac12$ for large values of $\tau$. This is not in contradiction to~\eqref{eq:FEM:discretisation_error:uqmu} due to the factor $\sqrt{1+\tau}$, which is not apparent in the numerical results. %
    Moreover, this reduced rate of $q-\nicefrac12$ is only visible for large values of $h$ since the factor $\tfrac{C_2 h\tau}{C_3 + h\tau} \leq C(\tau)h$ leads to an asymptotic rate of $q+\nicefrac12$ for fixed $\tau$.
    More precisely, for large values of $\tau$ the respective second term of~\eqref{eq:NE:eoc:2} and~\eqref{eq:NE:eoc:3} dominates if the mesh-width $h$ is not small, where the respective first term dominates if the mesh-width becomes small enough. This leads to zones of different local convergence rates for $\vq_h$ and $\Div \vq_h$ when $\tau$ is fixed. This shall be investigated in the following subsection.
    The jump of $\vq\cdot\vn$ tends to $0$ like $\tau$ for $\tau \to 0$ and fixed mesh-width $h$, where for small and fixed $\tau$ we observe a rate of $q+1$ in $h$, while it is $q$ for large values $\tau$. The jump of $\pitr u_h$ tends to $0$ like $\nicefrac{1}{\tau}$ for $\tau \to \infty$ and fixed
    mesh-width $h$, where for fixed $\tau$ we observe a rate of $q+1$ in $h$ for small $\tau$ and a rate of $q$ for large $\tau$.

\subsection{Convergence behavior for fixed $\tau$}
\begin{figure}[tb]
  \input{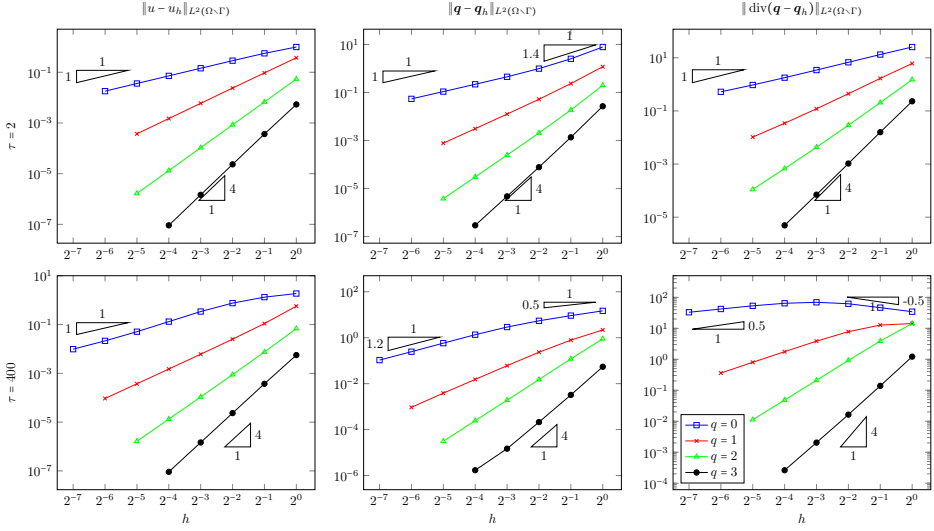}
  \caption{Error plots for $u_h$, $\mathbf{q}_h$ and $\mathrm{div}\mathbf{q}_h$ on $\Omega\setminus\Gamma$ for order $q=0,\,1,\,2,\,3$ and $\tau=2$ (upper row) and $\tau =400$ (lower row). The slope triangles indicate the slope corresponding to the behavior expected by~\eqref{eq:NE:eoc}, and the rates of $1.4$ and $1.2$ for the error in $q_h$ are exceptionally the observed slopes.}
  \label{NE:fig:cQ3r_regimes_upper}
\end{figure}

\begin{figure}[tb]
  \input{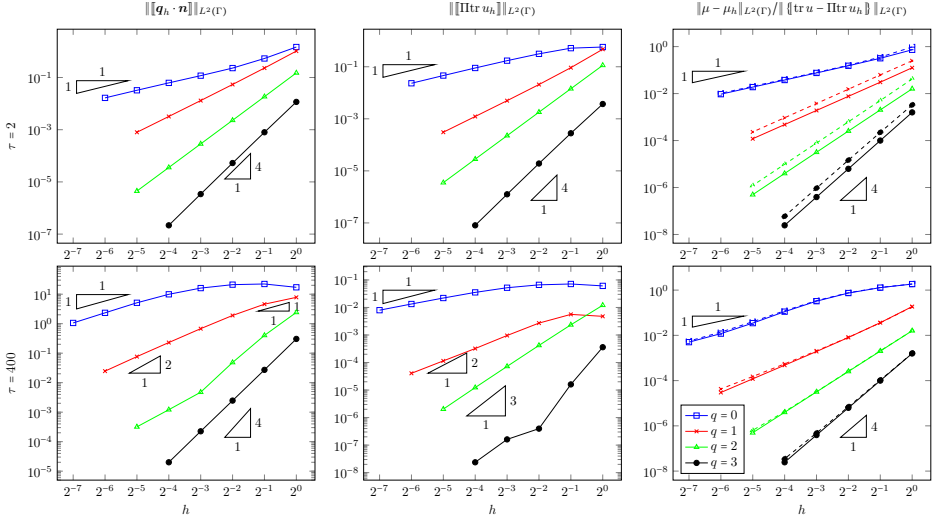}
  \caption{Error plots for $\jump{\mathbf{q}_h\cdot\mathbf{n}}$, $\jump{\pitr u_h}$, $\mu_h$ (solid), $\mean{\pitr u_h}$ (dashed) on $\Gamma$ for order $q=0,\,1,\,2,\,3$ and $\tau=2$ (upper row) and $\tau=400$ (lower row). The slope triangles indicate the slope corresponding to the behavior expected by~\eqref{eq:NE:eoc}.}
  \label{NE:fig:cQ3r_regimes_lower}
\end{figure}

To illustrate the change of different local convergence rates we have performed numerical experiments for $\tau = 2$ and $\tau = 400$. The results are shown in Fig.~\ref{NE:fig:cQ3r_regimes_upper} and Fig.~\ref{NE:fig:cQ3r_regimes_lower}.
For the $L^2(\OmG)$-error of $u_h$ we observe the same local convergence rates of $q+1$ for both values of $\tau$. %
We see a reduced convergence rate $q+\nicefrac12$ for the $L^2(\OmG)$-error of $\vq_h$ for $\tau = 400$ and $q = 0$ for larger mesh-widths and a convergence rate higher than~$1$ for smaller mesh-widths. For $\tau = 2$ and $q = 0$ we see a higher convergence rate than~$1$ for large mesh-widths and it tends to $1$ for smaller values. For higher values of $q$ we do not see a change of the convergence behavior so clearly as for $q = 0$, and a convergence rate of $q+1$ is observed for smaller mesh-widths.  For fixed value of $\tau$ we expect to see three zones, a local convergence of $q+\nicefrac12$ changes to more than $q+1$ for smaller mesh-width while changing to $q+1$ when the mesh is further refined.
For the $L^2(\OmG)$-error of $\Div \vq_h$ we see for $\tau = 400$ and $q = 0$ a reduced rate of $-\nicefrac12$ which turns into a rate of $\nicefrac12$ for smaller mesh-widths. For $\tau = 2$ we see for $q = 0$ a convergence rate of $1$, where we expect in view of~\eqref{eq:NE:eoc:3} an asymptotic rate of $\nicefrac12$ for $h \to 0$. Again, for higher $q$ a change of the convergence rates is not that clearly visible, but an optimal rate of $q+1$.
For the jump of $\vq_h \cdot \vn$ we observe convergence rates of $q+1$ for $\tau = 2$, where the error levels are lower than for $\tau = 400$ when comparing the same mesh-width. For $\tau = 400$ we see reduced rates for coarse meshes, which tends to $q+1$ when the mesh-width becomes small enough.
A similar behavior is observed for the jump of $\pitr u_h$, where the error levels are lower for $\tau = 400$ than for $\tau = 2$. For $\tau = 400$ the local convergence rates are lower on coarse meshes and turn to the optimal rates on finer meshes. For the error in $\mu_h$ or the mean of $\pitr u_h$ we observe the same convergence rates of $q+1$ for $\tau = 2$ and $\tau = 400$.

\section{Conclusion \& Outlook}
In this work, the HMDD method, a mixed domain decomposition formulation with Lehrenfeld–Schöberl stabilization terms inspired by hybridized discontinuous Galerkin methods, was considered. A corresponding consistent and well-posed variational formulation was identified in which the primal variable and the divergence of the dual variable are sums of $L^2$ functions on the subdomains and $L^2$ distributions on the skeleton depending on the stabilization parameter $\tau$. %
A conforming finite element discretization with Raviart-Thomas elements for the dual variable and piecewise polynomials for the primal and hybrid variables was shown to belong to a class of perturbed Galerkin methods whose well-posedness was proven analogously to the variational formulation of HMDD. The discretization error has been analyzed and investigated in numerical experiments in terms of the stabilization parameter $\tau$. %
The focus of this work was in the numerical analysis of the entire HMDD formulation. However, the well-posedness of patch problems whose degrees of freedom can be eliminated as well as the symmetry, positive definiteness and sparsity of the condensed system on the skeleton should be transferable from in HDG to HMDD.

The application to triangular meshes in two dimensions or hexahedral or tetrahedral meshes in three dimensions seems be straightforward, where the proof for well-posedness relies only on the existence of a continuous projection operator $(H^{\nicefrac12}(\OmG))^d \cap H(\Div,\OmG) \to \wideparen{\cW}$ as defined in Lemma~\ref{lem:HDD:Galerkin:deRham}. %
The extension to other discretization methods inside the patches like in isogeometric analysis or generalized FEM seems to be straightforward.

The presented mixed domain decomposition formulation with stabilization parameter $\tau > 0$ can be potentially applied in the mortar setting of non-matching grids~\cite{Arbogast.Cowsar.Wheeler.Yotov:2000,Arbogast.Pencheva.Wheeler.Yotov:2007} where the introduction of the variational framework shall prove useful.
To work with the variational framework might be essential for the application to further problems like the Helmholtz equation~\cite{Modave.ChaumetFrelet:2023} or for magnetostatics or more general electromagnetic problems.

\section*{Acknowledgment}
This work was supported by Time-X and by the Graduate School CE within the Profile Topic Computational Engineering at Technische Universität Darmstadt. %
We thank Mario Mally and Michael Wiesheu, funded by the Collaborative Research Centre – TRR361/F90: CREATOR, for their fruitful discussions and helpful insights.%
The second author has used AI tools throughout the paper for spelling and grammar checks.

\bibliographystyle{siamplain}
\bibliography{article}

\end{document}